\documentclass[11pt]{article}
\usepackage{amssymb,amsmath,amsthm,mathrsfs}
\usepackage[margin=.75in]{geometry}
\usepackage{enumitem}
\setlist{topsep=0pt,parsep=0pt,partopsep=0pt,itemindent=0pt}
\usepackage{tikz}
\usetikzlibrary{arrows,calc}
\tikzset{>=stealth}

\usepackage{indentfirst}
\usepackage[title]{appendix}
\usepackage{etoolbox}
\patchcmd{\appendices}{\quad}{: }{}{}
\title{Presentations for the Euclidean Picard modular groups}
\author{David Polletta\thanks{Author partially supported by National Science Foundation Grant DMS-1708463.} \protect \\
School of Mathematical and Statistical Sciences \protect\\
Arizona State University \protect \\
Tempe, Arizona, United States \protect \\
email: david.polletta@asu.edu}
\date{\today}
\begin{document}
\newgeometry{top= 0.75in, bottom=.75in, right=.75in, left=.75in}


\newcommand{\A}{\mathbb{A}}
\newcommand{\B}{\mathbb{B}}
\newcommand{\C}{\mathbb{C}}
\newcommand{\F}{\mathbb{F}}
\newcommand{\K}{\mathbb{K}}
\newcommand{\M}{\mathbb{M}}
\newcommand{\N}{\mathbb{N}}
\renewcommand{\P}{\mathbb{P}}
\newcommand{\Q}{\mathbb{Q}}
\newcommand{\R}{\mathbb{R}}
\newcommand{\Z}{\mathbb{Z}}
\newcommand{\card}{\operatorname{card}}

\newcommand{\U}{\operatorname{U}}
\newcommand{\GL}{\operatorname{GL}}
\newcommand{\PU}{\operatorname{PU}}
\newcommand{\SL}{\operatorname{SL}}
\newcommand{\SU}{\operatorname{SU}}
\newcommand{\PGL}{\operatorname{PGL}}
\newcommand{\Char}{\operatorname{Char}}

\newcommand{\gal}{\operatorname{Gal}}
\newcommand\aut[2]{\operatorname{Aut}_{{}_{#1}} ({#2})}
\newcommand\irr[3]{\operatorname{Irr}_{{}_{#1}}({#2},{#3})}

\newcommand{\Ann}{\operatorname{Ann}}
\newcommand{\Ch}{\operatorname{Ch}}
\newcommand{\coim}{\operatorname{coim}}
\newcommand{\coker}{\operatorname{coker}}
\newcommand{\Hom}{\operatorname{Hom}}

\renewcommand\bar[1]{\overline{#1}}
\newcommand\Id{\operatorname{Id}}
\newcommand\im{\operatorname{im}}
\newcommand\lcm{\operatorname{lcm}}
\newcommand\nr[2]{\operatorname{N}_{{}_{#1}} ({#2})}
\newcommand\tr[2]{\operatorname{Tr}_{{}_{#1}} ({#2})}

\newcommand\ds[1]{{\displaystyle #1}}
\newcommand\mc[1]{\mathcal{#1}}
\newcommand\mf[1]{\mathfrak{#1}}
\newcommand\ms[1]{\mathscr{#1}}
\newcommand\ssty[1]{{\scriptstyle #1}}
\newcommand\sssty[1]{{\scriptscriptstyle #1}}

\theoremstyle{plain}
\newtheorem{thm}{Theorem}
\newtheorem{lemma}{Lemma}
\newtheorem{prob}{Problem}
\newtheorem{defn}{Definition}
\newtheorem{prop}{Proposition}
\newtheorem{cor}{Corollary}

\theoremstyle{definition}
\newtheorem{conj}{Conjecture}
\newtheorem*{ex}{Example}
\newtheorem{alg}{Algorithm}
\newtheorem{exc}{Problem}

\theoremstyle{remark}
\newtheorem*{remark}{Remark}
\newtheorem*{note}{Note}
\newtheorem{case}{Case}

\maketitle
\begin{abstract}
 Mark and Paupert devised a general method for obtaining presentations for arithmetic non-cocompact lattices, \(\Gamma\), in isometry groups of negatively curved symmetric spaces.  The method involves a classical theorem of Macbeath applied to a \(\Gamma\)-invariant covering by horoballs of the negatively curved symmetric space upon which \(\Gamma\) acts. In this paper, we will discuss the application of their method to the Picard modular groups, PU\((2,1;\mathcal{O}_{d})\), when \(d=2,11\), and obtain presentations for these groups, which completes the list of presentations for Picard modular groups whose entries lie in Euclidean domains, namely those with \(d=1,2,3,7,11\).
\end{abstract}

\section{Introduction}
\let\thefootnote\relax\footnote{\textit{2010 Mathematics Subject Classification.} Primary; 22E40 Secondary; 32Q45.}There has been a great deal of study devoted to discrete subgroups and lattices in semisimple Lie groups.  In particular, the study of arithmetic lattices, which can be roughly described as lattices obtained by taking matrices with entries lying in the integer ring of some number field.  Some examples of arithmetic lattices are the \textit{Picard modular groups}, PU\((2,1;\mathcal{O}_{d})\), where \(\mathcal{O}_{d}\) represents the ring of integers in the number field \(\Q(i\sqrt{d})\), where \(d\) is a positive, square-free integer.  It is well known that if \(d \equiv 1,2 \mod (4)\), then \(\mathcal{O}_{d} = \Z[i\sqrt{d}]\), and if \(d \equiv 3 \mod (4)\), then \(\mathcal{O}_{d} = \Z[\frac{1+i\sqrt{d}}{2}]\).  It is also known that \(\mathcal{O}_d\) is a Euclidean domain exactly when \(d=1,2,3,7,11\).  Falbel and Parker derived a presentation for PU\((2,1;\mathcal{O}_{3})\) in \cite{3}, and Falbel, Francsics, and Parker obtained a presentation for PU\((2,1;\mathcal{O}_{1})\) in \cite{2}.  These presentations were obtained by constructing explicit fundamental domains for the action of \(\PU(2,1;\mathcal{O}_3)\) and \(\PU(2,1;\mathcal{O}_1)\) respectively on the complex hyperbolic plane and appealing to the Poincar\'e polyhedron theorem.  Complex hyperbolic space does not have totally geodesic real hypersurfaces like in real hyperbolic space, which makes constructing these fundamental domains rather challenging.  For \(d=1,3\), there are additional rotational symmetries making fundamental domain constructions more feasible, whereas these additional symmetries are not present for the remaining values of \(d\).  Mark and Paupert developed a different method to obtain a presentation for PU\((2,1;\mathcal{O}_{7})\), and also applied their method to the cases \(d=1,3\) \cite{8}.  Zhao obtained generators for the Euclidean Picard modular groups in \cite{12}, but did not derive sets of relations for these groups.  The covering argument (Section \ref{horcov}) used in \cite{8} and this paper is closely related to the argument Zhao used to obtain generators for the Euclidean Picard modular groups.  The Picard modular groups are generalizations of another class of arithmetic lattices known as the \textit{Bianchi groups}, \(\PGL(2,\mathcal{O}_d)\), where \(\mathcal{O}_d\) is the ring of integers previously defined.  The Bianchi groups are lattices in the isometry group of real hyperbolic 3-space, and Swan was able to derive presentations in \cite{11} by constructing fundamental domains for the action of these groups and applying a generalization of Macbeath's theorem. 
\newline \indent The main tool for obtaining the presentation for PU\((2,1;\mathcal{O}_{7})\), and the presentations in this paper, comes from a theorem of Macbeath which gives a presentation for a group, \(\Gamma\), acting on a topological space, \(X\), given an open subset \( V \subset X\) such that the \(\Gamma\)-translates of \(V\) cover \(X\).  Obtaining generators and relations for \(\Gamma\) roughly amounts to tracking pairwise and triple intersections of the set \(V\) and its translates.  In our setting, \(X\) is complex hyperbolic 2-space, \(\mathbb{H}^2_{\C}\), and \(\Gamma\) is a cusped arithmetic lattice in Isom(\(\mathbb{H}^2_{\C})\).  The fundamental set, \(V\), used in our application of Macbeath's theorem is a particular type of convex open set in \(\mathbb{H}_{\C}^2\) known as a \textit{horoball} (Definition \ref{hrball}).  These horoballs are anchored at points on the boundary at \(\infty\) of \(X\), \(\partial_{\infty}X\), and we will take \(V\) to be based at the point at \(\infty\) to yield some favorable conditions.  Since the complex hyperbolic plane is a four-dimensional space with a three-dimensional boundary, we are able to construct explicit pictures for phenomena occuring on \(\partial_{\infty}X\).  We will construct such pictures to aid in our argument that the \(\Gamma\)-translates of \(V\) cover \(X\) (Figures 1-4).  It is too difficult to prove directly that our collection of horoball translates cover \(X\).  Instead, we argue that it is sufficient to cover \(\partial V\) in order to get a covering of all of \(X\) (Lemma \ref{horcov}).  We also reduce the difficulty of the covering argument by showing that a covering of \(\partial V\) by open balls for a particular distance function known as the \textit{extended Cygan metric} (Section \ref{sieg}) implies a covering of \(\partial V\) by horoball translates.
\newline \indent These explicit pictures, however, give us no information about pairwise and triple intersections of the \(\Gamma\)-translates of \(V\) in the space \(X\).  In order to track pairwise intersections of translates of \(V\), we use an algebraic property, the notion of \textit{level} (Definition \ref{level}) between two \(E\)-rational points in \(\partial_{\infty}X\), and some desirable properties of basing \(V\) at \(\infty\).  The ability to use levels between rational boundary points comes from the fact that \(\Gamma\) is an \textit{integral lattice}, that is, \(\Gamma\) is contained in \(\PGL(n+1,\mathcal{O}_E)\) for the ring of integers, \(\mathcal{O}_E\), in some number field, \(E\).  Tracking triple intersections of translates of the set \(V \subset X\) is much more difficult in practice.  When searching for pairwise intersections of \(\Gamma\)-translates of \(V\), the notion of level and the consequences of having \(V\) based at \(\infty\), allow us to obtain an explicit description of the generating set of \(\Gamma\).  Unfortunately, we do not have an analagous algebraic criterion to track triple intersections of translates of \(V\) because now we are dealing with three horoballs, two of which do not have the optimal position based at \(\infty\).  Instead, we will show that if a triple intersection occurs, the anchors of the three horoballs of interest cannot be too far apart with respect to the Cygan metric (Lemma \ref{cybounds}).  In practice, this necessary condition may generate some additional relations, but these redundancies will later be eliminated.  
\newline \indent We will derive presentations for  PU\((2,1;\mathcal{O}_{d})\), \(d=2,11\), which completes the list of presentations for Picard modular groups where \(\mathcal{O}_d\) is a Euclidean domain.  In addition to the values of \(d\) where \(\mathcal{O}_d\) is a Euclidean domain, \(\mathcal{O}_{d}\) is a principal ideal domain, but not Euclidean, for the values \(d=19,43,67,163\).  Mark and Paupert's method can be applied to these cases as well, but as \(d\) grows, so does the complexity and computational cost of applying their method, making these presentations seem unattainable for the moment.  In the case \(d=2\), the method produces a presentation with a generating set of 54 matrices and 5,837 relations.  With the aid of MAGMA \cite{7}, we are able to obtain a presentation with 3 generating matrices and 29 relations (Theorem \ref{d=2 pres}).  In the case \(d=11\), we get a presentation with 263 generating matrices and 23,673 relations and MAGMA \cite{7} reduces the generating set to 5 matrices and the size of the relation set to 26 (Theorem \ref{d=11 pres}).  I would like to thank Alice Mark and my advisor, Julien Paupert, for their aid throughout my derivation of these presentations.  I am truly grateful for Dr. Mark and Dr. Paupert teaching me their method and their continuous assistance in its implementation.
\\[1\baselineskip]
Before we explain the method in practice, we summarize some relevant background information about complex hyperbolic space.  More details on complex hyperbolic space can be found in \cite{1}, \cite{4}.

\section{Complex hyperbolic space}
\subsection{Basic definitions}\label{comp}
Let \(\C^{n,1}\) denote the vector space \(\C^{n+1}\) equipped with a Hermitian form \(\langle \cdot, \cdot \rangle\) of signature \((n,1)\).  Denote \(V^{-} = \{Z \in \C^{n+1} : \langle Z,Z \rangle < 0 \}\), \(V^{0} = \{Z \in \C^{n+1} : \langle Z,Z \rangle = 0 \}\), and let \( \pi : \C^{n,1} - \{0\} \longrightarrow \C\P^{n}\) denote projectivization.  We define \(\mathbb{H}^{n}_{\C}\) to be \( \pi(V^{-}) \subset \C\P^{n}\).  We say a matrix \(A \in \GL(n+1,\C)\) is \textit{unitary} for the form \(\langle \cdot , \cdot \rangle\), if \(\langle A(Z),A(W) \rangle = \langle Z,W \rangle\) for all \(Z,W \in \C^{n+1}\). We denote the subgroup of unitary matrices by U\((n,1)\).  \(\mathbb{H}^{n}_{\C}\) comes equipped with a distance function, the Bergman metric, given by

\[ \cosh^{2}(\frac{d(\pi(Z),\pi(W))}{2}) = \frac{ |\langle Z, W \rangle|^{2}}{\langle Z,Z \rangle \langle W,W \rangle}\textrm{ } \textrm{ for }Z,W \in V^{-}\]
Note that the Bergman metric is independent of the choice of lifts for \(Z\) and \(W\).  Unitary matrices clearly preserve the Bergman metric, and it is well known that \(\textrm{Isom}^{0}(\mathbb{H}^{n}_{\C}) = \PU(n,1)\), where \(\textrm{Isom}^{0}(\mathbb{H}^{n}_{\C})\) denotes the identity component of \(\textrm{Isom}(\mathbb{H}^{n}_{\C})\), and \(\PU(n,1) = \U(n,1)/ \U(1)\) denotes the projective unitary group associated to the form \(\langle \cdot,\cdot \rangle\).  We also identify the boundary at infinity, \(\partial_{\infty}\mathbb{H}^{n}_{\C}\), with \( \pi(V^{0})\).  We will be dealing with the case when \(n=2\).  While there are multiple models for complex hyperbolic space, the \textit{Siegel model} of complex hyperbolic space is most useful for our calculations.

\subsection{The Siegel model for \(\mathbb{H}^{2}_{\C}\) and the Heisenberg group}\label{sieg}
The Siegel model of complex hyperbolic 2-space is defined as the projective model explained in Section \ref{comp} associated to the Hermitian form on
\(\C^{3}\) given by
\(\langle Z,W \rangle = W^{*}JZ\), where \({}^*\) denotes conjugate transpose and

\[J = \begin{bmatrix}
       0 & 0 & 1 \\
       0 & 1 & 0 \\
       1 & 0 & 0 \\
      \end{bmatrix} \]
The complex hyperbolic plane can then be parameterized by \(\C \times \R \times \R^{+}\) in the following way:
\[\mathbb{H}^{2}_{\C} = \{ \pi(\psi(z,v,u)) : z \in \C, v \in \R, u \in \R^{+}\}\]
Above, \(\pi\) is the projectivization map of Section \ref{comp} and \(\psi(z,v,u)\) is given by:
\begin{equation}\label{stdlft} 
\psi(z,v,u) = \begin{bmatrix}
                  \frac{-|z|^2 - u + iv}{2} \\
                  z \\
                  1 \\
                 \end{bmatrix}
\end{equation}
When we parameterize \(\mathbb{H}^2_{\C}\) in this way, \(\partial_{\infty}\mathbb{H}^{2}_{\C}\) becomes the one-point compactification:
\[\{ \pi(\psi(z,v,0)) : z \in \C, v \in \R\} \cup \{\infty\}\]
We will denote \(\infty = \pi((1,0,0)^{T})\), and call \((z,v,u) \in \C \times \R \times \R^{+}\) \(\textit{horospherical coordinates}\) of the point \( \pi(\psi(z,v,u)) \in \mathbb{H}^{2}_{\C}\).

\begin{defn}\label{hrball}
For a fixed \(u_{0} \in \R^{+}\), the level set \(H_{u_{0}} = \{ \pi(\psi(z,v,u_{0})) : z \in \C, v \in \R\}\) is called the horosphere at height \(u_{0}\) based at \(\infty\).  We call \(B_{u_{0}} = \{ \pi(\psi(z,v,u)) : z \in \C, v \in \R, u > u_{0}\}\) the open horoball at height \(u_{0}\) based at \(\infty\).
\end{defn}

We can identify the punctured boundary, \(\partial_{\infty}\mathbb{H}^{2}_{\C} - \{\infty\}\), with the \(\textit{Heisenberg group}\), \textbf{H}.  The Heisenberg group has \(\C \times \R\) as its underlying set and obeys the group law given by:
\begin{equation}\label{heis}
 (z_1,v_1)(z_2,v_2) = (z_1+z_2,v_1+v_2+2\textrm{Im}(z_1\cdot\bar{z_2}))
 \end{equation}
In the line above, ``Im'' denotes the imaginary component of a complex number and ``\(\cdot\)'' denotes the ordinary multiplication on \(\C\).  We can identify \(\partial_{\infty}\mathbb{H}^{2}_{\C} - \{\infty\}\) with \textbf{H} since \textbf{H} acts simply-transitively on \(\partial_{\infty}\mathbb{H}^{2}_{\C} - \{\infty\}\).  The action of \((z_1,v_1) \in \textbf{H}\) on an element \((z_2,v_2,0) \in \partial_{\infty}\mathbb{H}^2_{\C} - \{\infty\}\) is given by multiplying the vector \(\psi(z_2,v_2,0)\) on the left by the following matrix in U\((2,1)\): 

\begin{equation}\label{trans}
T_{(z_1,v_1)} = \begin{bmatrix}
                   1 & -\bar{z_1} & \frac{-|z_1|^2 + iv_1}{2} \\
                   0 & 1 & z_1\\
                   0 & 0 & 1 \\
                  \end{bmatrix}
\end{equation}
Matrices of the form above are known as \textit{Heisenberg translation matrices}.  It is an easy exercise to show that Heisenberg translations preserve a distance function on \textbf{H} known as the \(\textit{Cygan metric}\).  The expression for the Cygan metric for \((z_1,v_1),(z_2,v_2) \in \textbf{H}\) is given by:
\begin{equation}\label{cymet} d_{C}((z_1,v_1),(z_2,v_2)) = ||z_1 - z_2|^{4} + |v_1 - v_2 +2\textrm{Im}(z_1\cdot\bar{z_2})|^{2}|^{\frac{1}{4}} = |2\langle \psi(z_1,v_1,0),\psi(z_2,v_2,0) \rangle|^{\frac{1}{2}}
\end{equation}

The Cygan metric is the restriction to the punctured boundary of the complex hyperbolic plane of an incomplete distance function on
\(\bar{\mathbb{H}^{2}_{\C}} - \{\infty\}\) called the \(\textit{extended Cygan metric}\).  The expression for the extended Cygan metric for \((z_1,v_1,u_1),(z_2,v_2,u_2) \in \bar{\mathbb{H}^{2}_{\C}} - \{\infty\}\) is given by:
\begin{equation}\label{excymet}
\begin{split}
d_{XC}((z_1,v_1 , u_1),(z_2,v_2 , u_2)) &= |(|z_1 - z_2|^{2}+|u_1 - u_2|)^{2} + |v_1 - v_2 +2\textrm{Im}(z_1\cdot\bar{z_2})|^{2}|^{\frac{1}{4}}\\
& = |2\langle \psi(z_1,v_1,u_1),\psi(z_2,v_2,u_2) \rangle|^{\frac{1}{2}}
\end{split}
\end{equation}
(Extended) Cygan balls and spheres are defined in the obvious way.  We refer the reader to \cite{5} for more information on the extended Cygan metric.  In addition to Heisenberg translations, matrices of the form
\begin{equation}\label{rot}
A = \begin{bmatrix}
       1&0&0\\
       0&e^{i\theta}&0\\
       0&0&1\\
      \end{bmatrix}
\end{equation}
preserve the Cygan metric as well.  Matrices of this type are called \(\textit{Heisenberg rotation matrices}\).  Indeed, one can easily check that \(A((z,v)) = (e^{i\theta}z,v)\) and the claim follows since \(|e^{i\theta}|=1\).  There is also a connection between extended Cygan spheres and \textit{Ford isometric spheres}, which we define below.
\begin{defn}\label{ford}
 The Ford isometric sphere, \(I_g\), of an isometry \(g \in \textrm{U}(2,1)\) is the set:
 \[I_g = \{w = (z,v,u) \in \mathbb{H}^{2}_{\C} : |\langle \psi(w) , \psi(\infty)\rangle| = |\langle \psi(w), g^{-1}\psi(\infty)\rangle| \}\]
\end{defn}
\noindent By Proposition 4.3 of \cite{5} we have:
\begin{lemma}[\cite{5}]\label{Cyganford}
 Let \(g \in \textrm{U}(2,1)\) satisfy \(g(\infty) \neq \infty\).  Letting \(S\) denote the extended Cygan sphere with center \(g^{-1}(\infty)\) and radius \(\sqrt{2/|g_{3,1}|}\), we have \(S = I_{g}\).
\end{lemma}
When we apply Macbeath's theorem later, we will argue that the images under \(\PU(2,1;\mathcal{O}_{d})\) of a particular horoball of certain height based at \(\infty\) cover \(\mathbb{H}^2_{\C}\).  We will also make use of the following fact in our covering argument, which is Lemma 1 of \cite{2}.

\begin{lemma}[\cite{2}]\label{convex}
Extended Cygan balls are affinely convex in horospherical coordinates. 
\end{lemma}

\subsection{Lattices and isometries of \(\mathbb{H}^{2}_{\C}\)}
We say a subgroup \(\Gamma < \textrm{Isom}(\mathbb{H}^{2}_{\C})\) is a \textit{lattice} if \(\Gamma\) is a discrete subgroup of Isom(\(\mathbb{H}^{2}_{\C}\)) and \(\mathbb{H}^{2}_{\C}/\Gamma\) has finite volume.  We distinguish between the cases where \(\mathbb{H}^{2}_{\C}/\Gamma\) is compact and non-compact, and call the corresponding lattices cocompact and non-cocompact respectively.  The distinction comes down to whether or not \(\Gamma\) contains isometries of a certain type.
\\[1\baselineskip]
Elements of Isom\((\mathbb{H}^{2}_{\C})\) can be roughly characterized into three categories:\newline
1.\textit{Elliptic Isometries} - isometries that have a fixed point in \(\mathbb{H}^{2}_{\C}\). \newline
2.\textit{Parabolic Isometries} - isometries that do not have a fixed point in \(\mathbb{H}^{2}_{\C}\), but fix exactly one point of \(\partial_{\infty} \mathbb{H}^{2}_{\C}\).\newline
3.\textit{Loxodromic Isometries} - isometries that do not have a fixed point in \(\mathbb{H}^{2}_{\C}\), but fix exactly 2 points of \(\partial_{\infty} \mathbb{H}^{2}_{\C}\).
\par The well-known Godement compactness criterion states that a lattice in a semisimple Lie group defined over \(\Q\) contains parabolic elements if and only if it is non-cocompact.  We define a \(\textit{cusp point}\) of a lattice, \(\Gamma <\) Isom(\(\mathbb{H}^2_{\C}\)), to be a point of \(\partial_{\infty} \mathbb{H}^{2}_{\C}\) fixed by a parabolic element of \(\Gamma\), and a \(\textit{cusp group}\) of \(\Gamma\) to be a subgroup of the form \(\textrm{Stab}_{\Gamma}(p)\) where \(p \in  \partial_{\infty} \mathbb{H}^{2}_{\C}\) is a cusp point of \(\Gamma\).  A well-known result of Zink tells us that since \(\Q(i\sqrt{2})\) and \(\Q(i\sqrt{11})\) have class number one, the Picard modular groups for \(d=2,11\) both have a single cusp \cite{13}.

\section{Primitive integral lifts and levels}
\subsection{Primitive integral lifts}
The Picard modular groups, PU\((2,1,\mathcal{O}_{d})\), are examples of \(\textit{integral lattices}\), since the entries of these matrices come from the ring of integers, \(\mathcal{O}_{d}\).  We define a \(\textit{unitary integral lattice}\) to be a lattice that is contained in \(U(H,\mathcal{O}_{E})\) for some number field \(E\), with ring of integers
\(\mathcal{O}_{E}\), and Hermitian form \(H = \langle \cdot,\cdot \rangle\) defined over \(E\).  We say that an integral vector, \(P_{0} = (p_1,p_2,p_3)^{T} \in \mathcal{O}^{3}_{d}\) is \(\textit{primitive}\) if it has no non-trivial integral submultiples, that is, if \(\lambda^{-1}P_0 \in \mathcal{O}^{3}_{d}\) for some \(\lambda \in \mathcal{O}_{d}\), then \(\lambda\) is a unit in \(\mathcal{O}_{d}\).  If \(p\) is an \(\mathcal{O}_{d}\)-rational point in \(\C\P^2\), that is, the image under projectivization of a vector \(P = (p_1,p_2,p_3)^{T}\), a \(\textit{primitive integral lift}\) of \(p\) is any lift \(P_{0}\) of \(p\) to \(\mathcal{O}^{3}_{d}\) which is a primitive integral vector.  The following ``uniqueness'' of primitive integral lifts is stated without proof in \cite{8} as Lemma 1.  We provide a proof below for completeness.

\begin{lemma}\label{liftunique}
If \(\mathcal{O}_{d}\) is a principal ideal domain, then primitive integral lifts are unique up to multiplication by a unit.
\end{lemma}

\begin{proof}
 Let \(p\) be an \(\mathcal{O}^{3}_{d}\)-rational point of \(\C\P^{2}\), and let \(P_1\), \(P_2\) be two primitive integral lifts of \(p\).  As \(P_1\) and \(P_2\) represent two lifts of \(p\), we have
 \[P_1 = \lambda P_2\] for some \(\lambda \in \C-\{0\}\).  Denote \(P_1 = (p_{11},p_{12},p_{13})^T\) and \(P_2 = (p_{21},p_{22},p_{23})^T\) where each \(p_{ij} \in \mathcal{O}_{d}\), and \(1\leq i\leq2\), \(1\leq j\leq3\).  As \(P_1\) and \(P_2\) are lifts of an element of \(\C\P^2\), \(P_1\) and \(P_2\) have at least one pair of corresponding non-zero entries.  This implies \(\lambda \in \Q(i\sqrt{d})-\{0\}\), and since \(\mathcal{O}_{d}\) is a principal ideal domain, we can write \(\lambda = \frac{\lambda_1}{\lambda_2}\) for some \(\lambda_1,
 \lambda_2 \in \mathcal{O}_{d}\) with gcd\((\lambda_1,\lambda_2) = 1\).  Now, since \(\mathcal{O}_{d}\) is a principal ideal domain, it is also a unique factorization domain.  This implies we can write \(p_{ij} = u_{ij}p_{ij1}...p_{ijn_{ij}}\) and \(\lambda_m = v_{m}\lambda_{m1}...\lambda_{mn_{m}}\) where \(n_{ij},n_{m} \in \N\), \(p_{ijk}, \lambda_{mh} \in \mathcal{O}_{d}\) are primes, and \(u_{ij}, v_{m}\) are units for \(1\leq i,m \leq 2\), \(1 \leq j \leq 3\), \(1\leq k \leq n_{ij}\), and \(1\leq h \leq n_{m}\).
 \\[1\baselineskip]
 From the equation \(P_1 = \lambda P_2\), and the decompositions above, we can write
 \[v_2\lambda_{21}...\lambda_{2n_{2}}u_{1j}p_{1j1}...p_{1jn_{1j}} = 
  v_1\lambda_{11}...\lambda_{1n_{1}}u_{2j}p_{2j1}...p_{2jn_{2j}}
 \]
 Consider any \(\lambda_{2h_{2}}\), where \(1 \leq h_{2} \leq n_{2}\).  \(\lambda_{2h_{2}}\) is a prime that divides the right side of the equation above.  \(\lambda_{2h_{2}}\) cannot divide \(v_1\) or \(u_{2j}\), as these are units, and \(\lambda_{2h_{2}}\) cannot divide any \(\lambda_{1h_{1}}\), \(1\leq h_1 \leq n_1\) since gcd\((\lambda_1,\lambda_2) = 1\).  These statements imply that \(\lambda_{2h_{2}}\) must divide \(p_{2jk}\) for some \(1\leq k\leq n_{2j}\).  Since this holds for \(1\leq j \leq 3\), and \(P_2\) is a primitive integral lift, we have that \(\lambda_{2h_{2}}\) is a unit.  Since this holds for each \(1 \leq h_{2} \leq n_{2}\), \(\lambda_2\) is a unit, and \(\lambda\) is a member of \(\mathcal{O}_{d}\).  By the definition of primitive integral lifts, \(\lambda\) must be a unit as well.
 \end{proof}
 It turns out that any column-vector of a unitary matrix with entries from \(\mathcal{O}_{d}\) is a primitive integral vector.  Moreover, under certain criteria, unitary matrices send isotropic primitive integral vectors to primitive integral vectors.  Lemma 2 from \cite{8} tells us the following:
 
 \begin{lemma}[\cite{8}]\label{matrix}
 Any column-vector of a matrix \(A \in U(2,1;\mathcal{O}_{d})\) is a primitive integral vector.  Moreover, if \(\mathcal{O}_{d}\) is a principal ideal domain and one of the basis vectors, \(e_{i}\), is \(\langle \cdot , \cdot \rangle\)-isotropic, then for any primitive integral  \(\langle \cdot , \cdot \rangle\)-isotropic vector, \(P\), and any \(A \in U(2,1;\mathcal{O}_{d})\), \(AP\) is a primitive integral vector.
\end{lemma}

\subsection{Levels and depths}
\noindent Next, we define the level of \(\mathcal{O}_{d}\)-rational points on the boundary at infinity.

\begin{defn}\label{level}
 Given two \(\mathcal{O}_{d}\)-rational points, \(p,q \in \partial_{\infty} \mathbb{H}^{2}_{\C}\), the \(\underline{level}\) between \(p\) and \(q\), denoted \(lev(p,q)\), is \(|\langle P_{0},Q_{0} \rangle|^{2}\) for any two primitive integral lifts \(P_{0}, Q_{0}\) of \(p,q\) respectively.  When we are given a preferred \(\mathcal{O}_{d}\)-rational point \(\infty \in \partial_{\infty} \mathbb{H}^{2}_{\C}\), the \(\underline{depth}\) of an \(\mathcal{O}_{d}\)-rational \(p \in \partial_{\infty} \mathbb{H}^{2}_{\C}\) is the level between \(p\) and \(\infty\).
\end{defn}

By Lemma \ref{liftunique}, we see that the level between two \(\mathcal{O}_{d}\)-rational points is well defined when \(\mathcal{O}_d\) is a principal ideal domain.  Moreover, unitary matrices preserve levels since levels are determined by the form \(\langle \cdot,\cdot \rangle\).  Levels will allow us to find the maximal height, \(u\), of a horosphere \(H_{u} = \partial B_{u}\) based at \(\infty \in \partial_{\infty} \mathbb{H}^{2}_{\C}\), such that the \(\Gamma\)-translates of \(B_u\) cover \(\mathbb{H}^{2}_{\C}\).  This relies on the following result, which is Proposition 1 of \cite{8}.

\begin{lemma}[\cite{8}]\label{intersect}
 Let \(g \in \textrm{U}(2,1)\) satisfy \(g(\infty) \neq \infty\), \(S = I_{g^{-1}}\) (Definition \ref{ford}), and \(H_{u_{0}}\) the horosphere based at \(\infty\) of height \(u_{0}>0\).  Then \(H_{u_0} \cap gH_{u_0} = H_{u_0} \cap S\).
\end{lemma}

\noindent Lemma \ref{intersect} gives us the following corollary, which is listed as Corollary 1 in \cite{8}.  A proof is provided in \cite{8}, and we will flesh out some additional details below.

\begin{cor}[\cite{8}]\label{covdepth}
 For any \(\mathcal{O}_{d}\)-rational point \(p \in \partial_{\infty} \mathbb{H}^{2}_{\C}\) with depth \(n \geq1\), and any integral matrix \(A_{p} \in \PU(2,1;\mathcal{O}_{d})\) satisfying \(A_p(\infty) = p\), the set \(H_{u} \cap A_{p}(H_{u})\) is empty if and only if \(u > \frac{2}{\sqrt{n}}\).
\end{cor}

\begin{proof}
As \(A_p(\infty) = p\), and \(e_1 = (1,0,0)^T\) is a lift of \(\infty\), the first column of the unitary matrix associated to \(A_p\) is a primitive integral lift, \(P_0\), of \(p\), by Lemma \ref{matrix}.  This implies that the depth, \(n\), of \(p\) is \(|\langle P_0,e_1 \rangle|^2 = |{A_p}_{3,1}|^2\).  Since \(dep(p) \geq 1\), necessarily \(A_p(\infty) = p \neq \infty\) (\(\infty\) has depth zero).  Letting \(S = I_{A_{p}^{-1}}\), by Lemma \ref{Cyganford}, we have that the radius of \(S\) equals \((\frac{4}{n})^{\frac{1}{4}}\).  By Lemma \ref{intersect}, \(H_{u} \cap A_{p}^{-1}(H_{u}) = H_{u} \cap S \), which implies
\begin{align*}
 H_{u} \cap A_{p}(H_{u}) \neq \emptyset &\iff H_{u} \cap A_{p}^{-1}(H_u) \neq \emptyset \iff H_{u} \cap S \neq \emptyset\\ 
 &\iff \exists(z,v,u) \in H_u \textrm{ s.t. } d_{XC}(w,p) = (\frac{4}{n})^{\frac{1}{4}} \\
 &\iff \exists (z,v,u) \in H_u \textrm{ s.t. } |(|z - p_z|^{2}+|u - 0|)^{2} + |v - p_v +2Im(z\cdot\bar{p_z})|^{2}| = \frac{4}{n} \\
 & \iff |u|^2 \leq \frac{4}{n} \\
 & \iff u \leq \frac{2}{\sqrt{n}}
 \end{align*}
 Above, we used the equation for the extendend Cygan distance, the fact that \(u>0\), and the horospherical coordinates \((p_z,p_v,0)\) for \(p\).
\end{proof}

\begin{defn}\label{coveringdepth}
 The \underline{covering height} of \(\Gamma = \PU(2,1;\mathcal{O}_{d})\), denoted \(u^{cov}\), is the maximal height such that \(\Gamma B_{u^{cov}}\) covers \(\mathbb{H}^2_{\C}\).  The \underline{covering depth} of \(\Gamma\) is the unique \(n \in \N\) such that \(\frac{2}{\sqrt{n+1}} < u^{cov} \leq \frac{2}{\sqrt{n}}\).
\end{defn}

\section{Macbeath's Theorem}
\noindent The main tool for applying this method comes from a classical result of Macbeath which we now state.

\begin{thm}[Macbeath's Theorem \cite{6}]
Let \(\Gamma\) be a group acting by homeomorphisms on a topological space \(X\).  Let \(V\) be an open subset of \(X\) whose \(\Gamma\)-translates cover \(X\).\newline
1. If \(X\) is connected, then the set \(E(V) = \{\gamma \in \Gamma : V \cap \gamma V \neq \emptyset\}\) generates \(\Gamma\). \newline
2. If \(X\) is also simply-connected and \(V\) is path-connected, then \(\Gamma\) admits a presentation with generating set \(E(V)\) and relations \(\gamma \cdot \gamma ' = \gamma \gamma'\) for all \(\gamma ,\gamma ' \in E(V)\) such that \(V \cap \gamma V \cap \gamma \gamma ' V \neq \emptyset\).
\end{thm}

We state the result above as written in \cite{6}.  Macbeath's theorem tells us that \(\Gamma\) admits a presentation, \(\langle S:R \rangle\), where \(S = \{e_{\gamma} : \gamma \in E(V)\}\) and \(R = \{e_{\gamma} \cdot e_{\gamma'} = e_{\gamma\gamma'}\}\), where ``\(\cdot\)'' is the word product and \(\gamma,\gamma',\gamma\gamma'\) are elements of \(E(V)\) satisfying the triple intersection property.  In our context, \(\Gamma = \PU(2,1;\mathcal{O}_{d})\),  \(X = \mathbb{H}^2_{\C}\), and \(V = B_{u^{cov}}\).  For the relevant calculations, we will use the basepoint \(\infty \in \partial_{\infty} \mathbb{H}^{2}_{\C}\) for our cusp point representative, and use \(\infty = \pi((1,0,0)^{T})\) in the Siegel model of complex hyperbolic 2-space (Section \ref{sieg}).  Below, we collect the relevant objects and notation used in our derivations throughout Section 4.

\subsubsection*{Notation}
\begin{itemize}
 \item \(\Gamma\): The lattice, \(\PU(2,1;\mathcal{O}_{d})\)
 \item \(B\): The fundamental open horoball based at \(\infty\), \(V = B_{u^{cov}}\)
 \item \(\infty\): A fixed cusp representative, \(\pi((1,0,0)^T)\)
 \item \(\Gamma_{\infty}\): The cusp stabilizer, \(\textrm{Stab}_{\Gamma}(\infty)\), with presentation \(\langle S_{\infty}:R_{\infty} \rangle\)
 \item \(n\): The covering depth of \(\Gamma\)
 \item \(u^{cov}\): The covering height corresponding to covering depth of \(\Gamma\) (Definition \ref{coveringdepth})
 \item \(D_{\infty}\): A compact fundamental domain for action of \(\Gamma_{\infty}\) on \(H_{u^{cov}} \simeq \partial_{\infty} \mathbb{H}^{2}_{\C} - \{\infty\}\)
 \item \(p_i\): A representative of the \(\Gamma_{\infty}\)-orbit of an \(\mathcal{O}_d\)-rational point of at most depth \(n\) in \(D_{\infty}\)
 \item \(A_i\): An element of \(\Gamma\) satisfying \(A_i(\infty) = p_i\)
 \item \(\gamma_{\infty}^{j}\): an element of \(\Gamma_{\infty}\) (\(j\) is a superscript, not an exponent)
\end{itemize}

\noindent Note, since \(\Gamma\) has a single cusp in the cases \(d=2,11\), each \(p_i\) is in the \(\Gamma\)-orbit of \(\infty\), and it is possible, in theory, to find such a matrix satisfying \(A_i(\infty) = p_i\).

\subsection{Covering \(\mathbb{H}^2_{\C}\)}\label{horcov}
\begin{prop}\label{covarg}
 Let \(B_u\) denote the open horoball of height \(u>0\) based at \(\infty\).  If the \(\Gamma\)-translates of \(B_u\) cover \(H_u = \partial B_u\), then the \(\Gamma\)-translates of \(B_u\) cover \(\mathbb{H}^2_{\C}\).
\end{prop}

\begin{proof}
 Denote \(C = \bigcup \limits_{\gamma \in \Gamma} \gamma B_u\) is nonempty.  We will show \(C\) is both open and closed in \(\mathbb{H}^2_{\C}\), hence \(C = \mathbb{H}^2_{\C}\).  By construction, \(C\) is nonempty.  Since \(C\) is a union of open sets, \(C\) is open.  To show \(C\) is closed, let \(x \in \bar{C}\).  There exist sequences \((\gamma_n)\) in \(\Gamma\) and \((x_n)\) in \(B_u\) such that \(\gamma_nx_n \rightarrow x\). Since we have a compact fundamental domain, \(D_{\infty}\), for the action of \(\Gamma_{\infty}\) on \(H_u = \partial B_u\), we need only finitely many horoball translates to cover \(D_{\infty}\) and we can choose a uniform \(\epsilon>0\) so that the \(\epsilon\)-neighborhood of \(D_{\infty}\) is covered by translates of \(B_u\).  As \(\Gamma\) acts on \(\mathbb{H}^2_{\C}\) by isometries, and every open horoball intersecting \(H_u\) is a \(\Gamma\)-translate of one of the horoballs covering \(D_{\infty}\), we can extend the \(\epsilon\)-neighborhood of \(D_{\infty}\) to all of \(H_u\).  Using the original horoball, \(B_u\), along with the \(\epsilon\)-neighborhood of \(H_u\), we get an \(\epsilon\)-neighborhood of \(B_u\).  Again, since \(\Gamma\) acts on \(\mathbb{H}^2_{\C}\) by isometries, it follows that the \(\epsilon\)-neighborhood of any translate of \(B_u\) is also covered by translates of \(B_u\).  Since \(\gamma_nx_n \rightarrow x\), we have \(d(\gamma_nx_n,x) < \epsilon\) for sufficiently large \(n\), so that \(x\) is contained in an \(\epsilon\)-neighborhood of \(\gamma_nx_n\) and \(x \in C\).  Thus, \(C\) is closed and \(C = \mathbb{H}^2_{\C}\).
 \end{proof}
Proposition \ref{covarg} tells us we need only cover \(H_u = \partial B_u\) by \(\Gamma\)-translates of \(B_u\) to cover \(\mathbb{H}^2_{\C}\).  By Lemma \ref{intersect}, we know \(H_{u} \cap A_i^{-1}(H_{u}) = H_{u} \cap S\), where \(S\) is the extended Cygan sphere with center \(p_i\) and radius \((\frac{4}{dep(p_i)})^{\frac{1}{4}}\).  This implies we can determine that translates by \(\Gamma\) of \(B_{u}\) cover \(\mathbb{H}^2_{\C}\) by obtaining a cover of \(H_{u}\) by extended Cygan balls centered on \(\partial_{\infty}\mathbb{H}^2_{\C}\).  Since we have a compact fundamental domain, \(D_{\infty} \subset H_{u^{cov}}\), for the action of \(\Gamma_{\infty}\) on \(H_{u^{cov}} \simeq \partial_{\infty} \mathbb{H}^{2}_{\C} - \{\infty\}\), we need only cover \(D_{\infty}\) by extended Cygan balls centered on \(\mathcal{O}_{d}\)-rational points in \(\partial_{\infty} \mathbb{H}^{2}_{\C}\).  We will obtain this covering of \(D_{\infty}\) by decomposing \(D_{\infty}\) into affine pieces, each given as a convex hull of a finite number of points in \(D_{\infty}\).  We will argue that these pieces cover the entirety of \(D_{\infty}\), and that the vertices defining each convex hull piece is contained in an extended Cygan ball centered on an \(\mathcal{O}_{d}\)-rational point in \(\partial_{\infty} \mathbb{H}^{2}_{\C}\).  By Lemma \ref{convex}, since extended Cygan balls are affinely convex, the entire convex hull defined by the set of vertices must be in the same extended Cygan ball.  Thus \(D_{\infty}\) is covered by extended Cygan balls centered on \(\partial_{\infty}\mathbb{H}^2_{\C}\), and we get a cover of \(\mathbb{H}^2_{\C}\) by the \(\Gamma\)-translates of \(B_u\).

\subsection{Generators}\label{gens}
\begin{prop}\label{gener}
 Suppose the \(\Gamma\)-translates of \(B\) cover \(\mathbb{H}_{\C}^2\).  Then \(\Gamma\) is generated by \(S_{\infty} \cup \mathcal{A}\), where \(S_{\infty}\) is the generating set for \(\Gamma_{\infty}\), and \(\mathcal{A}\) is a finite collection of matrices sending \(\infty\) to points of depth at most \(n\) in \(D_{\infty}\).
\end{prop}
\begin{proof}
By Macbeath's theorem, \(\gamma\) is a generator for \(\Gamma\) if \(B \cap \gamma B \neq \emptyset\).  Corollary \ref{covdepth} tells us that \(B \cap \gamma B \neq \emptyset\) if and only if \(\gamma\infty\) is either \(\infty\), or an \(\mathcal{O}_d\)-rational point of depth at most \(n\).  The proof of Proposition \ref{covarg} tells us that we need only cover \(D_{\infty}\) by finitely many horoballs in order to obtain our covering of \(\mathbb{H}_{\C}^2\).  The bases of these horoballs give us a finite list, \(\{p_1,...,p_k\}\), of \(\mathcal{O}_d\)-rational points in \(D_{\infty}\).  Assume the first \(r\) of the points form a system of representatives for \(\{p_1,...,p_k\}\) under the action of \(\Gamma_{\infty}\).  Since \(\Gamma\) has a single cusp, we can find a matrix, \(A_i\), sending \(\infty\) to \(p_i\).  If \(A_i\) is such a matrix, then \(A_i\) sends \(B\) to the horoball based at \(p_i\) as well.  We also know that any \(\mathcal{O}_d\)-rational point of depth at most \(n\) is a \(\Gamma_{\infty}\)-translate of one of the points \(\{p_1,...,p_r\}\).  This implies:
\[E(B) = \{\gamma \in \Gamma : B \cap \gamma B \neq \emptyset \} = \{\gamma^{1}_{\infty}A_i\gamma^{2}_{\infty} : \gamma^{1}_{\infty}, \gamma^{2}_{\infty} \in \Gamma_{\infty}, i = 1,...,r\}\]
The previous statement accounts for multiplying \(A_i\) on the left by an element of \(\Gamma_{\infty}\), and the right-multiplication of \(A_i\) by elements of \(\Gamma_{\infty}\) comes from the fact that precomposing a matrix with an element of the cusp stabilizer does not change the image of \(\infty\) under the map \(A_i\).  Denoting \(\mathcal{A} = \{A_1,...,A_r\}\), we conclude \(\Gamma\) is generated by \(S_{\infty} \cup \mathcal{A}\).
\end{proof}

\subsection{Relations}
It is not straightforward to check if a triple intersection of horoballs is nonempty.  Instead, we settle for some necessary conditions for a triple intersection to occur.  Relaxing the criteria for a possible relation may generate additional relations, but these redundancies will later be eliminated.  Recall our generating set is \(E(B) = \{\gamma^{1}_{\infty}A_i\gamma^{2}_{\infty} : \gamma^{1}_{\infty}, \gamma^{2}_{\infty} \in \Gamma_{\infty}, i = 1,...,r\}\).  Macbeath's theorem tells us that \(\Gamma\) admits a relation, \(\gamma\cdot \gamma' = \gamma \gamma'\), whenever \(\gamma, \gamma' \in E(B)\) satisfy \(B \cap \gamma B \cap \gamma \gamma ' B \neq \emptyset\).  Suppose both \(\gamma\infty \neq \infty\) and \(\gamma\gamma'\infty \neq \infty\).  As elements of \(\Gamma_{\infty}\) stabilize \(B\), and left and right multiplication by elements of \(\Gamma_{\infty}\) permute the set \(E(B)\), we may assume
\(\gamma = A_a, \gamma' = \gamma^{1}_{\infty}A_b\gamma^{2}_{\infty}, \textrm{ and } \gamma\gamma' = \gamma^{3}_{\infty}A_c\gamma^{4}_{\infty}\) for some \(1 \leq a,b,c \leq r\) and \(\gamma^{i}_{\infty}\in \Gamma_{\infty}\), \(1 \leq i \leq 4\).  The corresponding relation \(\gamma\cdot\gamma' = \gamma\gamma'\) becomes \(A_a\gamma^{1}_{\infty}A_b\gamma^{2}_{\infty} = \gamma^{3}_{\infty}A_c\gamma^4_{\infty}\).  If \(\gamma\), \(\gamma'\), or \(\gamma\gamma'\) fix \(\infty\), we take the corresponding \(A\) matrix to be the identity.  When we evaluate both sides of the relation at \(\infty\) we get \(A_a\gamma^{1}_{\infty}p_b = \gamma^{3}_{\infty}p_c\), and we will use this equation to detect potential relations.  We obtain a relation, \(R_{a,b,c}\), by identifying the element \(A^{-1}_c(\gamma^{3}_{\infty})^{-1}A_a\gamma^{1}_{\infty}A_b = \gamma_{\infty}^{*} \in \Gamma_{\infty}\) as a word composed of elements of the generating set \(S_{\infty}\).  The equation  \(A_a\gamma^{1}_{\infty}p_b = \gamma^{3}_{\infty}p_c\) provides us some useful information regarding the Cygan distance of our points of interest that will aid in our search for possible relations.

\begin{lemma}\label{covdist}
 If \(A_a\gamma^{1}_{\infty}p_b = \gamma^{3}_{\infty}p_c\) for \(A_a \in \Gamma\), \(\gamma_{\infty}^1,\gamma_{\infty}^3 \in \Gamma_{\infty}\), and \(p_b,p_c \in D_{\infty}\), then
 \[d_{C}({A_{a}}^{-1}(\infty), \gamma_{\infty}^{1}p_{b}) = \left(\frac{4dep(p_{c})}{dep(p_{a})dep(p_{b})}\right)^{\frac{1}{4}}\]
 where \(p_a = A_a(\infty)\).
\end{lemma}
\begin{proof}
 Suppose \(A_a\gamma^{1}_{\infty}p_b = \gamma^{3}_{\infty}p_c\).  As \(A_{a}\gamma_{\infty}^{1}p_{b}\) and \(\gamma_{\infty}^{3}p_{c}\) represent the same point of
 \(\partial_{\infty} \mathbb{H}^{2}_{\C}\), and unitary matrices preserve depths and levels, we have
\[A_{a} \gamma_{\infty}^{1}p_{b} = \gamma_{\infty}^{3}p_{c} \implies
dep(A_{a} \gamma_{\infty}^{1}p_{b}) = dep(\gamma_{\infty}^{3}p_{c}) \implies
lev(\infty, A_{a} \gamma_{\infty}^{1}p_{b}) = dep(p_{c})\]
\[\implies lev({A_{a}}^{-1}(\infty), \gamma_{\infty}^{1}p_{b}) = dep(p_{c})
\]
From Definition \ref{level}, we can also write
\[lev({A_{a}}^{-1}(\infty), \gamma_{\infty}^{1}p_{b}) = |\langle P,Q \rangle | ^{2}\]
where \(P\) and \(Q\) are primitive integral lifts of the points \(A_a^{-1}(\infty)\) and \(\gamma^1_{\infty}p_b\) respectively.  We can write 
\[lev({A_{a}}^{-1}(\infty), \gamma_{\infty}^{1}p_{b}) = dep(A_a^{-1}(\infty))dep(\gamma_{\infty}^1p_b)|\langle \psi(A_a^{-1}(\infty), \psi(\gamma_{\infty}^1p_b) \rangle|^{2}\]
where \(\psi(\cdot)\) represents the standard lift of a point in the Siegel model (eq. \ref{stdlft}).  By the remarks earlier, we have \(dep(\gamma^1_{\infty}p_b) = dep(p_b)\), and since unitary matrices preserve depths and levels,
\[dep(p_a) = lev(\infty,p_a) = lev(\infty, A_a(\infty)) = lev(A_a^{-1}(\infty), \infty) = dep(A_a^{-1}(\infty))\]
The statements above imply
\begin{equation*}
 dep(p_c) = lev({A_{a}}^{-1}(\infty), \gamma_{\infty}^{1}p_{b}) = dep(p_a)dep(p_b)|\langle \psi(A_a^{-1}(\infty), \psi(\gamma_{\infty}^1p_b) \rangle|^{2}
 \end{equation*}
 The expression for the Cygan metric (eq. \ref{cymet}) and the equation above imply
 \[d_{C}({A_{a}}^{-1}(\infty), \gamma_{\infty}^{1}p_{b}) = \left(\frac{4dep(p_{c})}{dep(p_{a})dep(p_{b})}\right)^{\frac{1}{4}}\]

\end{proof}
Lemma \ref{covdist} and the following lemma will show that if we have a triple intersection of \(\Gamma\)-translates of \(B\), the corresponding parabolic fixed points cannot be too far apart in the Cygan distance, and their distances are restrained by the covering depth of \(\Gamma\).  These facts will allow us to reduce our search for relations to a finite number of verifications.  Note, \(p_0 = (0,0,0) \in \partial_{\infty}\mathbb{H}^2_{\C}\) is an \(\mathcal{O}_{d}\)-rational point of depth 1 for all values of \(d\).
\begin{lemma}\label{cybounds}
 Let \(n\) be the covering depth of \(\Gamma\).  If the elements \(\gamma = A_a\), \(\gamma'=\gamma^1_{\infty}A_b\gamma^2_{\infty}\), \(\gamma\gamma'=\gamma^3_{\infty}A_c\gamma^4_{\infty}\) of \(\Gamma\) satisfy \(B \cap \gamma B \cap \gamma \gamma ' B \neq \emptyset\), and we obtain the relation \(A^{-1}_c(\gamma^{3}_{\infty})^{-1}A_a\gamma^{1}_{\infty}A_b = \gamma_{\infty}^{*}\), then \(d_C(\gamma^1_{\infty}p_0,p_0)\), \(d_C(\gamma^3_{\infty}p_0,p_0)\), and \(d_C(\gamma^*_{\infty}p_0,p_0)\) are bounded as follows:
 
 \begin{enumerate}
\item If \(p_a,p_b,p_c \neq \infty\) \newline
\(d_{C}(\gamma_{\infty}^{1}p_{0},p_0) \leq \max\limits_{1\leq i \leq r}d_{C}(p_{0},{A_{i}}^{-1}(\infty)) + 
(4n)^{\frac{1}{4}} + \max\limits_{1\leq i \leq r}d_{C}(p_{i},p_{0})\) \newline
\(d_C(\gamma^3_{\infty}p_0,p_0) \leq (4n)^{\frac{1}{4}} + 2\max\limits_{1\leq i \leq r}d_{C}(p_{i},p_{0})\) \newline
\(d_C(\gamma^*_{\infty}p_0,p_0) \leq (4n)^{\frac{1}{4}} + 2\max\limits_{1\leq i \leq r}d_{C}(A_{i}^{-1}(\infty),p_{0})\)
 \item If \(p_a,p_b \neq \infty\), \(p_c = \infty\) \newline
 \(d_{C}(\gamma_{\infty}^{1}p_{0},p_0) \leq \max\limits_{1\leq i \leq r}d_{C}(p_{0},{A_{i}}^{-1}(\infty))+ \max\limits_{1\leq i \leq r}d_{C}(p_{i},p_{0})\)\newline
 \(d_C(\gamma^*_{\infty}p_0,p_0) \leq \max\limits_{1\leq i \leq r}d_{C}(p_{0},{A_{i}}^{-1}(\infty))+ \max\limits_{1\leq i \leq r}d_{C}(p_{i},p_{0})\)
 \item If \(p_b = \infty\), \(A_a = A_c\) \newline
 \(d_C(\gamma^1_{\infty}p_0,p_0) \leq 2\max\limits_{1\leq i \leq r}d_{C}(p_{i},p_{0})\) \newline
 \(d_C(\gamma^*_{\infty}p_0,p_0) \leq 2\max\limits_{1\leq i \leq r}d_{C}(A_{i}^{-1}(\infty),p_{0})\)
 \end{enumerate}
\end{lemma}

\begin{proof}
 Suppose \(B \cap \gamma B \cap \gamma \gamma ' B \neq \emptyset\) and \(n \in \N\) is the covering depth of \(\Gamma\).  Recall, the triple intersection property gives us a relation of the form \(A^{-1}_c(\gamma^{3}_{\infty})^{-1}A_a\gamma^{1}_{\infty}A_b = \gamma^{*}_{\infty}\).
 \begin{case}[\(p_a,p_b,p_c \neq \infty\)]
 Evaluating both sides of \(A^{-1}_c(\gamma^{3}_{\infty})^{-1}A_a\gamma^{1}_{\infty}A_b = \gamma^{*}_{\infty}\) at \({\infty}\) yields \(A_{a} \gamma_{\infty}^{1}p_{b} = \gamma_{\infty}^{3}p_{c}\).  Lemma \ref{covdist} implies
\begin{equation}\label{eq.1}
d_{C}({A_{a}}^{-1}(\infty), \gamma_{\infty}^{1}p_{b}) = \left(\frac{4dep(p_{c})}{dep(p_{a})dep(p_{b})}\right)^{\frac{1}{4}}
\end{equation}
By the triangle inequality, (eq. \ref{eq.1}), and the fact that elements of \(\Gamma_{\infty}\) preserve the Cygan metric, we have
\begin{align*}
d_{C}(p_{0},\gamma_{\infty}^{1}p_{0}) \leq& d_{C}(p_{0},{A_{a}}^{-1}(\infty)) + 
d_{C}({A_{a}}^{-1}(\infty), \gamma_{\infty}^{1}p_{b}) + d_{C}(\gamma_{\infty}^{1}p_{b},\gamma_{\infty}^{1}p_{0}) \\
=& d_{C}(p_{0},{A_{a}}^{-1}(\infty)) + 
\left(\frac{4dep(p_{c})}{dep(p_{a})dep(p_{b})}\right)^{\frac{1}{4}} + d_{C}(p_{b},p_{0})
\end{align*}
As \( 1 \leq dep(p_a),dep(p_b),dep(p_c) \leq n\), and there are finitely many \(A_{i}\) and \(p_{i}\), we have
\[d_{C}(p_{0},\gamma_{\infty}^{1}p_{0}) \leq \max\limits_{1\leq i \leq r}d_{C}(p_{0},{A_{i}}^{-1}(\infty)) + 
(4n)^{\frac{1}{4}} + \max\limits_{1\leq i \leq r}d_{C}(p_{i},p_{0})\]
Using the relation \(A^{-1}_c(\gamma^{3}_{\infty})^{-1}A_a\gamma^{1}_{\infty}A_b = \gamma^{*}_{\infty}\), and performing an identical procedure with the derived equations \({A_{a}}^{-1}\gamma_{\infty}^{3}p_{c} = \gamma_{\infty}^{1}p_{b}\) and \((\gamma^3_{\infty})^{-1}p_a =A_c \gamma^*_{\infty}A_{b}^{-1}(\infty)\), we obtain
\[d_C(\gamma^3_{\infty}p_0,p_0) \leq (4n)^{\frac{1}{4}} + 2\max\limits_{1\leq i \leq r}d_{C}(p_{i},p_{0}) \textrm{ and } 
d_C(\gamma^*_{\infty}p_0,p_0) \leq (4n)^{\frac{1}{4}} + 2\max\limits_{1\leq i \leq r}d_{C}(A_{i}^{-1}(\infty),p_{0})\]
\end{case}
\begin{case}[\(p_a,p_b \neq \infty\), \(p_c = \infty\)]
In this case, our relation is \(A_a\gamma^1_{\infty}A_b = \gamma^*_{\infty}\).  Note, we need not produce a \(\gamma^3_{\infty}\) in this case.  As \(p_c = \infty\), we have \(dep(p_c) = 0\).  This fact, and the work form Case 1 implies\begin{align*}
d_{C}(p_{0},\gamma_{\infty}^{1}p_{0}) \leq&d_{C}(p_{0},{A_{a}}^{-1}(\infty)) + 
\left(\frac{4dep(p_{c})}{dep(p_{a})dep(p_{b})}\right)^{\frac{1}{4}} + d_{C}(p_{b},p_{0})\\
=&d_{C}(p_{0},{A_{a}}^{-1}(\infty)) + d_{C}(p_{b},p_{0})
\end{align*}
so that
\[d_{C}(p_{0},\gamma_{\infty}^{1}p_{0}) \leq \max\limits_{1\leq i \leq r}d_{C}(p_{0},{A_{i}}^{-1}(\infty))+ \max\limits_{1\leq i \leq r}d_{C}(p_{i},p_{0})\]
From our relation, we can derive \(A_{a}\gamma^1_{\infty} = \gamma^*_{\infty}A_b^{-1}\).  Evaluating at \(\infty\) yields \(p_a = \gamma^*_{\infty}A_b^{-1}(\infty)\).  This implies
\begin{align*}
d_C(\gamma^*_{\infty}p_0,p_0) \leq& d_C(\gamma^*_{\infty}p_0,p_a) + d_C(p_a,p_0) \\
=& d_C(\gamma^*_{\infty}p_0,\gamma^*_{\infty}A_b^{-1}(\infty)) + d_C(p_a,p_0) \\
=& d_C(p_0,A_b^{-1}(\infty)) + d_C(p_a,p_0)
\end{align*}
so that
\[d_C(\gamma^*_{\infty}p_0,p_0) \leq \max\limits_{1\leq i \leq r}d_{C}(p_{0},{A_{i}}^{-1}(\infty))+ \max\limits_{1\leq i \leq r}d_{C}(p_{i},p_{0})\]
\end{case}

\begin{case}[\(p_b = \infty\), \(A_a = A_c\)]
In this case, our relation can be written \(A_a^{-1}\gamma^1_{\infty}A_a = \gamma^*_{\infty}\).  Note, we need not produce a \(\gamma^3_{\infty}\) in this case either.  Evaluating at \(\infty\) gives \(\gamma^1_{\infty}p_a = p_a\).  We have
\begin{align*}
 d_C(\gamma^1_{\infty}p_0,p_0) \leq& d_C(\gamma^1_{\infty}p_0,p_a) + d_C(p_a,p_0) \\
 =& d_C(\gamma^1_{\infty}p_0,\gamma^1_{\infty}p_a) + d_C(p_a,p_0) \\
 =& 2d_C(p_a,p_0)
\end{align*}
so that
\[d_C(\gamma^1_{\infty}p_0,p_0) \leq 2\max\limits_{1\leq i \leq r}d_{C}(p_{i},p_{0})\]
From our relation, we can derive \(A_a\gamma^*_{\infty}A_a^{-1} = \gamma^1_{\infty}\), so that evaluating at infinity gives \(\gamma^*_{\infty}A_a^{-1}(\infty) = A_a^{-1}(\infty)\).  Following an identical procedure, replacing \(\gamma^1_{\infty}\) and \(p_a\) with \(\gamma^*_{\infty}\) and \(A_a^{-1}(\infty)\) respectively, we obtain
\[d_C(\gamma^*_{\infty}p_0,p_0) \leq 2\max\limits_{1\leq i \leq r}d_{C}(A_{i}^{-1}(\infty),p_{0})\]
\end{case}
\noindent The remaining possible combinations of \(p_a = \infty\), \(p_b=\infty \), and \(p_c =\infty\) reduce to the cases already considered via appropriate left/right multiplication of matrices.
\end{proof}   

\subsection{Outline of method in practice}
\noindent Before moving on to the presentations for the Euclidean Picard modular groups in the cases \(d=2,11\), we will give a brief outline of the steps in applying Macbeath's theorem.
\begin{enumerate}
 \item Obtain an affine fundamental domain, \(D_{\infty} \subset \partial_{\infty}\mathbb{H}^2_{\C} \simeq H_{u^{cov}}\) for the action of \(\Gamma_{\infty} = \textrm{Stab}_{\Gamma}(\infty)\).
 \item Obtain a presentation \(\Gamma_{\infty} = \langle S_{\infty}:R_{\infty} \rangle\).
 \item Determine the covering depth, \(n\), of \(\Gamma\).
 \item Find all \(\mathcal{O}_{d}\)-rational points of depth at most \(n\) in \(D_{\infty}\).
 \item Of the \(\mathcal{O}_{d}\)-rational points of depth at most \(n\) in \(D_{\infty}\), find a system of representatives of these points under the action of \(\Gamma_{\infty}\).  Denote this system of representatives \(\{p_1,...,p_r\}\).
 \item For each \(p_i\), \(1 \leq i \leq r\), obtain a matrix, \(A_i\), such that \(A_i(\infty) = p_i\).  Denote this set of matrices \(\mathcal{A}\).
 \item For every \(1 \leq a,b,c \leq r\), for which there exists \(\gamma^1_{\infty},\gamma^3_{\infty} \in \Gamma_{\infty}\), such that \(A_a\gamma^1_{\infty}p_b = \gamma^3_{\infty}p_c\), record the relation \(R_{a,b,c} = (\gamma^*_{\infty})^{-1}A_c^{-1}(\gamma^3_{\infty})^{-1}A_a\gamma^1_{\infty}A_b\) for some \(\gamma^*_{\infty} \in \Gamma_{\infty}\).  Denote this set relations \(\mathcal{R}\).
 \item Obtain the presentation \(\Gamma = \langle S_{\infty}\cup \mathcal{A}: R_{\infty} \cup \mathcal{R} \rangle\).
 \item Simplify the presentation obtained in Step 8.
\end{enumerate}

\noindent Now that we have summarized the method, we move on to the presentations for \(\PU(2,1;\mathcal{O}_{d})\), \(d=2,11\).

\section{A presentation for \(\PU(2,1;\mathcal{O}_{2})\)}
 \begin{thm}\label{d=2 pres}
  \(\Gamma{(2)} = \PU(2,1;\mathcal{O}_{2})\) admits the following presentation:
  \small
\begin{align*}
&\Gamma(2) = \langle T,I,A :I^2,A^8,(T^{-1}A^{-1})^4,(T^2ITITI)^2,T^{-1}A^{-1}IATA^{-1}IA,
IA^3IA^{-1}IA^{-1},\\
&IT^{-1}A^{-2}IAT^{-1}ITA^2ITA,TATITIATATIT^{-1}IA,(T^{-1}A^{-2}T^{-1}IA^{-2})^2,(T^{-1}I)^8,\\
&(A^{-1}IA^{-1}T^{-1}A^{-1}IT^{-1})^2,IA^2ITATIA^{-1}T^{-1}A^{-1}IT^{-1}A,A^{-2}IA^{-1}T^{-1}IA^{-1}TIT^{-1}AITI,\\
&T^{-1}A^{-1}T^{-1}A^{-1}IT^{-1}AITIA^{-1}TIATA,IA^{-2}T^{-1}ITATIT^{-1}AITA^{-1}IT^{-1},\\
&TATITIT^{-1}IT^{-1}A^{-1}T^{-1}A^2IAI,A^2IA^{-2}TIT^{-1}AIA^{-1}IATIT^{-1},\\
&T^{-1}A^2IA^{-1}ITIT^{-1}A^{-1}IATIA^{-2}IA^{-1},TA^{-1}IAT^{-1}IA^{-1}T^{-1}IA^{-2}IAIA^{-1}TIA^{-1},\\
&T^{-2}A^2ITATIT^2IT^{-1}A^{-1}T^{-1}IA^{-2},ATAIATATITAIA^{-2}T^{-1}IA^{-1}T,\\
&A^{-1}T^{-1}AIT^{-1}ITIAIA^{-1}TITA^{-1}IA^2IT^{-1},T^{-1}IA^{-2}ITIA^{-2}IAT^{-1}IT^{-1}ITITIA^{-1},\\
&ATIT^{-1}IA^2IA^2T^{-1}A^{-1}T^{-1}IA^{-2}TIAIA^{-1}T,TATITITIT^{-1}IT^{-1}IT^{-1}A^{-1}T^{-1}AIT^{-1}ITIA^{-1},\\
&ITIA^{-1}T^{-1}IA^{-1}TIATAT^{-1}IATAIT^{-2}A^{-1}T^{-1}IA^{-2}T^{-3},\\
&IT^{-1}IT^{-1}IT^{-1}AIA^{-1}TIA^{-2}IAT^{-1}IT^{-1}A^{-1}TIT^{-1}AIA^{-1}IT^{-1}A^{-1}IT,\\
&TATIT^2ITA^{-1}TIT^{-1}A^{-1}T^{-1}A^{-1}T^2IT^{-1}IT^{-1}IT^{-1}IA^{-1}IT^{-1}ITIAIA^{-1}TA^{-2},\\
&T^{-1}ITA^{-2}TAIT^{-1}IT^{-2}ITA^{-1}IA^2IT^{-1}ITAIT^{-1}ITITIAIA^{-1}TITA^{-1}T^{-1}IA^{-2}T^{-1}\rangle\\
\end{align*}

\normalsize
The unitary matrices corresponding to the generators of \(\Gamma(2)\) are given by:
\[ T = T_{2\sqrt{2}} = \begin{bmatrix}
                        1 & 0 & i\sqrt{2} \\
                        0 & 1 & 0 \\
                        0 & 0 & 1 \\
                       \end{bmatrix},
I = I_0 = \begin{bmatrix}
         0 & 0 & 1 \\
         0 & -1 & 0 \\
         1 & 0 & 0 \\
        \end{bmatrix},
A = A_{3,3} = \begin{bmatrix}
         -1+i\sqrt{2} & 2i\sqrt{2} & 2-i\sqrt{2} \\
         2-i\sqrt{2} & 1-2i\sqrt{2} & -2 \\
         1-i\sqrt{2} & -i\sqrt{2} & -1 \\
        \end{bmatrix}\]
 \end{thm}
 
 \begin{cor}\label{ab2}
 The abelianization of \(\Gamma(2)\) is \(\Z/2\Z \times \Z/4\Z\).
\end{cor}
 
 \subsection{Finding \(D_{\infty}(2)\) and \(\Gamma_{\infty}(2)\)}
  \noindent More information on the cusp stabilizer, \(\Gamma_{\infty}(2)\), can be found in Section 5.3 of \cite{9}.
\begin{lemma}[\cite{9}]
1. The cusp stabilizer, \(\Gamma_{\infty}(2)\), admits the following presentation:
\[\Gamma_{\infty}(2) = \left\langle T_{2},T_{i\sqrt{2}}, T_{2\sqrt{2}},R:{T_{2\sqrt{2}}}^{4}[T_{2},T_{i\sqrt{2}}]\textrm{, } [T_{2},T_{2\sqrt{2}}],[T_{i\sqrt{2}},T_{2\sqrt{2}}], R^{2},[R,T_{2\sqrt{2}}],(RT_{2})^{2},(RT_{i\sqrt{2}})^{2} \right\rangle\]
2. Let \(D_{\infty}(2) \subset \partial_{\infty} \mathbb{H}^{2}_{\C}\) be the affine convex hull of the points with horospherical coordinates \((0,0),(2,0),(i\sqrt{2},0),\)
\((0,2\sqrt{2}),(2,2\sqrt{2}),(i\sqrt{2},2\sqrt{2})\).  Then \(D_{\infty}(2)\) is a fundamental domain for \(\Gamma_{\infty}(2)\) acting on \(\partial_{\infty} \mathbb{H}^{2}_{\C} - \{\infty\}\).
\\[1\baselineskip]
We use the following generators for \(\Gamma_{\infty}(2)\):

\[T_2 = \begin{bmatrix}        
         1 & -2 & -2 \\
         0 & 1 & 2 \\
         0 & 0 & 1 \\
        \end{bmatrix}\textrm{, }
        T_{i\sqrt{2}} =  \begin{bmatrix}        
         1 & i\sqrt{2} & -1 \\
         0 & 1 & i\sqrt{2} \\
         0 & 0 & 1 \\
        \end{bmatrix}\textrm{, }
        T_{2\sqrt{2}} = \begin{bmatrix}
                        1 & 0 & i\sqrt{2} \\
                        0 & 1 & 0 \\
                        0 & 0 & 1 \\
                       \end{bmatrix}\textrm{, }
                       R = \begin{bmatrix}
                            1 & 0 & 0 \\
                            0 & -1 & 0 \\
                            0 & 0 & 1 \\
                           \end{bmatrix}\]
where \(T_{z}\) is a Heisenberg translation matrix (eq. \ref{trans}) and \(R\) a Heisenberg rotation matrix (eq. \ref{rot}).
  \end{lemma}
\noindent The following lemma establishes a ``normal form'' for elements of \(\Gamma_{\infty}(2)\).  This normal form is crucial in our process for making our relation set finite.

\begin{lemma}\label{nmfm2}
 For any \(\gamma \in \Gamma_{\infty}(2)\), \(\gamma\) can be written as:
 \[ \gamma = R^{p}{T_{2\sqrt{2}}}^{n}{T_2}^{m}{T_{i\sqrt{2}}}^{l} \textrm{ for some } m,n,l \in \Z
 \textrm{ and } p = 0,1\]
\end{lemma}

\begin{proof}
One can easily verify that the relations of \(\Gamma_{\infty}(2)\) imply the following:
\begin{enumerate}
 \item \(T_{2\sqrt{2}}\) is central in  \(\Gamma_{\infty}(2)\).
 \item Commuting \(R\) with \(T_{2}\), \({T_{2}}^{-1}\), \(T_{i\sqrt{2}}\), or \({T_{i\sqrt{2}}}^{-1}\) changes \(T_{2}\) and \(T_{i\sqrt{2}}\) to its inverse and visa-versa.
 \item Commuting \({T_{2}}^{\pm 1}\) with \({T_{i\sqrt{2}}}^{\pm 1}\) introduces a factor of
 \({T_{2\sqrt{2}}}^{\pm 4}\) in the word decomposition of an element of \(\Gamma_{\infty}(2)\).
 \item \(R = R^{-1}\)
\end{enumerate}
Hence, given an element \(\gamma \in \Gamma_{\infty}(2)\), we can:
\begin{enumerate}
 \item Collect all \(R\) letters in the word decomposition of \(\gamma\) to the leftmost position and simplify.  Along the way, we either elminate a factors of \(R\), leave factors of \(T_{2\sqrt{2}}^{\pm 1}\) unchanged, or flip \(T_2^{\pm 1}\), \(T_{i\sqrt{2}}^{\pm 1}\) factors to their inverses.
 \item Collect all \(T_{2\sqrt{2}}^{\pm 1}\) letters to the second leftmost position in the word decomposition of \(\gamma\) using centrality of \(T_{2\sqrt{2}}\) and simplify.
 \item Collect all \(T_2^{\pm 1}\) letters in the third leftmost position in the word decomposition of \(\gamma\) and simplify.  Along the way, we potentially introduce factors of \(T_{2\sqrt{2}}^{\pm 4}\), we move these to the second leftmost position using centrality again.
 \item Simplify remaining collected letters of \(T_{i\sqrt{2}}^{\pm 1}\).
\end{enumerate}
The algorithm above gives us our desired normal form.
\end{proof}

\noindent Now that we have established a normal form for cusp stabilizer elements, we argue we need only check finitely many possible relations.
\begin{lemma}\label{expbds2}
 If \(\Gamma(2)\) admits the relation \(A^{-1}_c(\gamma^{3}_{\infty})^{-1}A_a\gamma^{1}_{\infty}A_b = \gamma^{*}_{\infty}\) for \(a,b,c \in \{1,...,r\}\) and \newline \(\gamma^1_{\infty}\),\(\gamma^3_{\infty}\),\(\gamma^{*}_{\infty} \in \Gamma_{\infty}(2)\), then, using the normal form as in Lemma \ref{nmfm2}, the exponents of \(\gamma^1_{\infty}\), \(\gamma^3_{\infty}\), and \(\gamma^{*}_{\infty}\) satisfy 
 \[|n|\leq 19 \textrm{, } |m|\leq 3 \textrm{, } |l|\leq 4\]
\end{lemma}

\begin{proof}
Recall the horospherical coordinates of the depth 1 point, \(p_0 = (0,0)\).  Lemma \ref{cybounds} tells us

\[d_{C}(\gamma_{\infty}^{1}p_{0},p_0) \leq \max\limits_{1\leq i \leq r}d_{C}(p_{0},{A_{i}}^{-1}(\infty)) + 
(4n)^{\frac{1}{4}} + \max\limits_{1\leq i \leq r}d_{C}(p_{i},p_{0})\]
We will see later, that the covering depth of \(\Gamma(2)\) is at most \(16\) (Section \ref{covdep2}).  Using the \(\mathcal{O}_2\)-rational points listed in Appendix \ref{2_points}, the bounds from Lemma \ref{cybounds} for the \(d=2\) case satisfy
\[\left(\frac{4dep(p_{c})}{dep(p_{a})dep(p_{b})}\right)^{\frac{1}{4}} \leq 2\sqrt{2} \textrm{, } 
d_{C}(p_{b},p_{0}) < 1.7048 \textrm{, }
d_{C}(p_{0},{A_{a}}^{-1}(\infty)) < 1.7684\]
Recall that the Cygan metric is given by (eq. \ref{cymet})
\[d_{C}((z_1,v_1),(z_2,v_2)) = ||z_1 - z_2|^{4} + |v_1 - v_2 +2\textrm{Im}(z_1\cdot\bar{z_2})|^{2}|^{\frac{1}{4}}\]
Using Lemma \ref{nmfm2}, and the Heisenberg multiplication law (eq. \ref{heis}), we have
\[d_C(p_0,\gamma^1_{\infty} p_0) = d_C(p_0, R^{p}{T_{2\sqrt{2}}}^{n}{T_2}^{m}{T_{i\sqrt{2}}}^{l}p_0) = d_C(p_0, {T_{2\sqrt{2}}}^{n}{T_2}^{m}{T_{i\sqrt{2}}}^{l}p_0) = ||2m+li\sqrt{2}|^{4} + |(2n-4ml)\sqrt{2}|^2|^{\frac{1}{4}} \]
The equations above imply
\[||2m+li\sqrt{2}|^{4} + |(2n-4ml)\sqrt{2}|^2|^{\frac{1}{4}} \leq 2\sqrt{2} + 1.7048+1.7684\]
Since \(l,m,n\) are integers, there are only finitely many value combinations that satisfy the inequality.  A simple calculation yields
\[|n|\leq 19 \textrm{, } |m|\leq 3 \textrm{, } |l|\leq 4\]

\noindent The derivations for the remaining cusp stabilizer elements follow the same procedure above.
\end{proof}

\subsection{Covering depth of \(\Gamma(2)\)}\label{covdep2}
\noindent Let \(B((z,v),r)\) be the open extended Cygan ball centered at \(p = (z,v) \in \partial_{\infty}\mathbb{H}^2_{\C}\) with radius \(r\).  Recall that balls of depth \(n\) appear at height \(u(n) =\frac{2}{\sqrt{n}}\).

\begin{lemma}
 Let \(u = u(17) + \epsilon =0.4852\), and \(H_u\) be the horosphere of height \(u\) based at \(\infty\).  Then the prism, \(D_{\infty}(2) \times \{u\}\), is covered by the intersections with \(H_u\) of the following extended Cygan balls (eq. \ref{excymet}): 
 \\[1\baselineskip]
 \small
 Depth 1: \(B((0,0),\sqrt{2}), B((2,0),\sqrt{2}), B((i\sqrt{2},0),\sqrt{2}), B((0,2\sqrt{2}),\sqrt{2}),B((2,2\sqrt{2}),\sqrt{2}), B((i\sqrt{2},2\sqrt{2}),\sqrt{2})\)
 \\[1\baselineskip]
 Depth 3: \(B((\frac{2}{3} + \frac{1}{3}i\sqrt{2}),\frac{2}{3}\sqrt{2}),(\frac{4}{3})^{\frac{1}{4}}), B((\frac{4}{3} + \frac{1}{3}i\sqrt{2},2\sqrt{2}),(\frac{4}{3})^{\frac{1}{4}})\)
 \normalsize
\end{lemma}

\begin{cor}
 The covering depth of \(\Gamma(2)\) is at most \(16\).
\end{cor}

\noindent Although we only used Cygan balls of depth up to 3, it appears that we still need to pass to depth 16.  We generated pictures (Figure 1) of coverings of \(D_{\infty}(2)\) by Cygan balls.  For heights corresponding to depths \(n \leq 15\), it appeared that balls of depth at most \(n\) did not cover the prism.  The proof below makes rigorous the fact that passing to depth \(16\) is sufficient in covering \(\mathbb{H}_{\C}^2\).

\begin{proof}
Figure 1 below shows the covering of \(D_{\infty}(2)\) by the relevant extended Cygan balls.  We will decompose \(D_{\infty}(2)\) into affine polyhedra, each of which lies in a single extended Cygan ball.  Consider the following points of \(\partial_{\infty}\mathbb{H}^2_{\C}\) in horospherical coordinates:
  \\[1\baselineskip]
  \small
  \(c_{1,1} = (0,0), c_{1,2} = (0,2\sqrt{2}), c_{1,3} = (2,0), c_{1,4} = (2,2\sqrt{2}), c_{1,5} = (i\sqrt{2},0), c_{1,6} = (i\sqrt{2},2\sqrt{2}), q_1 = (1,0),\)
  \\[1\baselineskip]
  \(q_2 =(\frac{3}{2}+\frac{1}{4}i\sqrt{2},0), q_3 = (1+\frac{1}{2}i\sqrt{2},0),q_4 = (\frac{1}{2}+\frac{3}{4}i\sqrt{2},0), q_5 =(1.15 +\frac{1}{4}i\sqrt{2},0),q_6 = (\frac{1}{2}i\sqrt{2},0), \)
  \\[1\baselineskip]
  \(q_7 = (\frac{1}{4}i\sqrt{2},0), q_8 = (1,2\sqrt{2}),q_9 = (1+\frac{1}{8}i\sqrt{2},2\sqrt{2}), q_{10} = (\frac{3}{2}+\frac{1}{4}i\sqrt{2},2\sqrt{2}), q_{11} = (1+\frac{1}{2}i\sqrt{2},2\sqrt{2}),\)
  \\[1\baselineskip]
  \(q_{12} =(\frac{1}{2}+\frac{3}{4}i\sqrt{2},2\sqrt{2}),q_{13} = (\frac{1}{2}i\sqrt{2},2\sqrt{2}), q_{14} = (\frac{1}{2}+\frac{1}{2}i\sqrt{2},2\sqrt{2}),q_{15} = (\frac{4}{5}+\frac{2}{5}i\sqrt{2},2\sqrt{2}), q_{16} = (1, \frac{15}{11}\sqrt{2}),\)
  \\[1\baselineskip]
  \(q_{17} = (1.1, \sqrt{2}), q_{18} = (2,\sqrt{2}), q_{19} = (0.94,\frac{3}{4}\sqrt{2}), q_{20} = (0, \frac{2}{3}\sqrt{2}), q_{21} = (0,\frac{5}{8}\sqrt{2}),q_{22} = (\frac{1}{4}i\sqrt{2},\frac{2}{3}\sqrt{2}),\)
  \\[1\baselineskip]
  \(q_{23} = (\frac{1}{2}i\sqrt{2},\sqrt{2}), q_{24} = (i\sqrt{2},\sqrt{2}), q_{25} = (\frac{3}{4} + \frac{5}{8}i\sqrt{2},\frac{1}{2}\sqrt{2}), q_{26} = (1.15+0.425i\sqrt{2},\frac{14}{15}\sqrt{2}),\)
  \\[1\baselineskip]
  \(q_{27} = (\frac{3}{4} +\frac{4}{11}i\sqrt{2},\sqrt{2}), q_{28} = (1+\frac{1}{4}i\sqrt{2},\frac{3}{2}\sqrt{2}), q_{29} = (1.07+\frac{2}{35}i\sqrt{2}, \frac{3}{4}\sqrt{2})\)
  \normalsize
  \\[1\baselineskip]
  Denoting \(C(S)\) the affine convex hull of a subset \(S \subset H_u \simeq \partial_{\infty}\mathbb{H}^2_{\C} \times \{u\}\), each of the following pieces of \(D_{\infty}(2) \times \{u\}\) are contained in the corresponding open extended Cygan ball:
  \\[1\baselineskip]
  \small
  \(R_1 = C(c_{1,1},q_1,q_5,q_3,q_7,q_{19},q_{21},q_{29}) \subset B((0,0),\sqrt{2})\) \newline
  \(R_2=C(c_{1,2},q_8,q_9,q_{15},q_{14},q_{13},q_{20},q_{16},q_{22},q_{23},q_{27},q_{28}) \subset B((0,2\sqrt{2}),\sqrt{2})\) \newline
  \(R_3 = C(c_{1,3},q_1,q_5,q_2,q_{19},q_{17},q_{18},q_{29}) \subset
  B((2,0),\sqrt{2})\) \newline
  \(R_4 = C(c_{1,4},q_2,q_3,q_5,q_8,q_9,q_{10},q_{18},q_{17},q_{16},q_{26},q_{28},q_{29}) \subset B((2,2\sqrt{2}),\sqrt{2})\) \newline
  \(R_5 = C(c_{1,5},q_4,q_6,q_{11},q_{12},q_{14},q_{15},q_{22},q_{23},q_{24},q_{25},q_{27}) \subset B((i\sqrt{2},0),\sqrt{2})\) \newline
  \(R_6 = C(c_{1,6},q_{12},q_{13},q_{14},q_{24},q_{23}) \subset B((2,2\sqrt{2}),\sqrt{2})\)\newline
  \(R_7 = C(q_3,q_4,q_6,q_7,q_{16},q_{17},q_{19},q_{20},q_{21},q_{22},q_{25},q_{26},q_{28},q_{27},q_{29}) \subset B((\frac{2}{3} + \frac{1}{3}i\sqrt{2}),\frac{2}{3}\sqrt{2})\) \newline
  \(R_8 = C(q_9,q_{10},q_{11},q_{15},q_{25},q_{26},q_{27},q_{28}) \subset B((\frac{4}{3} + \frac{1}{3}i\sqrt{2},2\sqrt{2}),(\frac{4}{3})^{\frac{1}{4}})\)
  \normalsize
  
   \begin{figure}[h]\label{fig1}
 \caption{Covering of \(D_{\infty}(2) \times \{u\}\) with Cygan balls up to depth 16}
 \includegraphics[scale = 0.2]{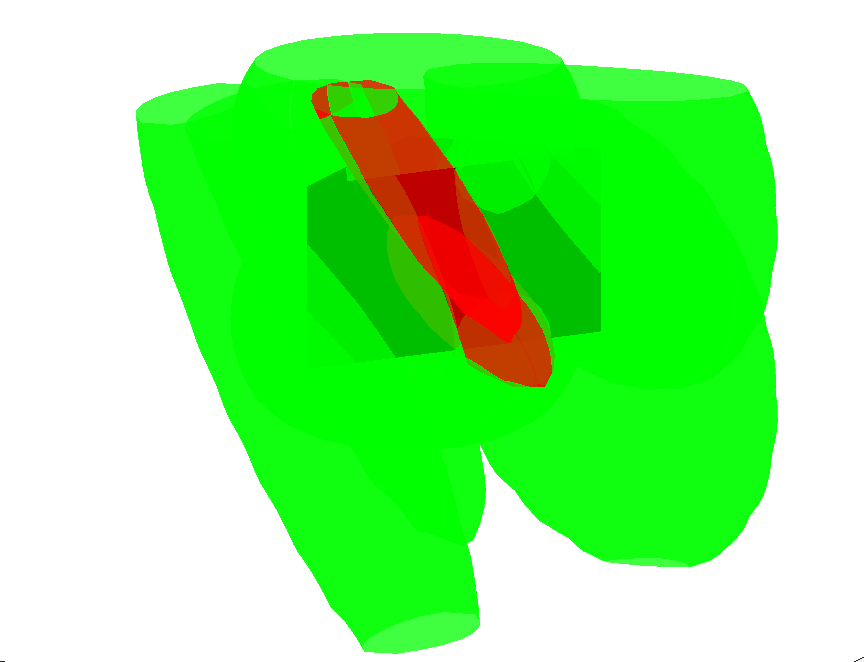}
 \centering
\end{figure}
  \par To verify the set containments above, we check each \(q\) in each region, \(R_i\), has an extended Cygan distance less than the radius of the corresponding extended Cygan ball.  Using the fact that extended Cygan balls are affinely convex, we can then conclude the entire convex hull is contained in \(R_i\).  We will do the calcuations for \(R_1\) here. The calculations for the remaining regions follow the same process.  We have:
  \\[1\baselineskip]
  \(d_C((0,0),c_{1,1}) < 0.6967 < \sqrt{2}\), \(d_C((0,0),q_1) < 1.2188 < \sqrt{2}\), \(d_C((0,0),q_5) < 1.3903 < \sqrt{2}\), \newline
  \(d_C((0,0),q_3) < 1.4091 < \sqrt{2}\),  \(d_C((0,0),q_7) < 0.7813 < \sqrt{2}\),
  \(d_C((0,0),q_{19}) < 1.3160 < \sqrt{2}\), \newline
  \(d_C((0,0),q_{21}) < 1.0042 < \sqrt{2}\),
  \(d_C((0,0),q_{29}) < 1.3966 < \sqrt{2}\)
  \\[1\baselineskip]
  The result then follows as each face of the regions, \(R_i\), are either on the boundary of the prism, or entirely contained in neighboring regions (see Figure 2).  Note, there is some overlap between some of the regions, but the entirety of the prism is still covered.
 \end{proof}

 \begin{figure}[h]\label{fig2}
 \caption{An affine decomposition of \(D_{\infty}(2) \times \{u\}\)}
 \includegraphics[scale = 0.23]{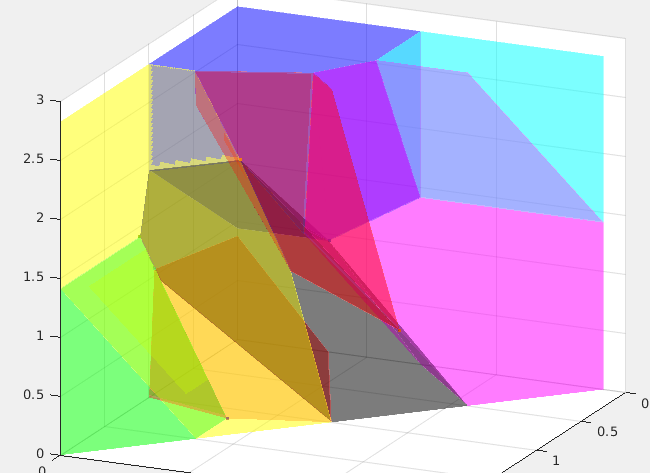}
 \centering
\end{figure}

\subsection{\(\mathcal{O}_2\)-rational points of depth at most \(n=16\)}

We aim to find all the points of depth at most \(16\) in \(D_{\infty}(2)\).  Recall that the depth of an \(\mathcal{O}_2\)-rational point, \(p\), is defined as the level between \(p\) and \(\infty = \pi((1,0,0)^{T})\).  We can take \( P_{\infty}=  (1,0,0)^{T}\) as a primitive integral lift for \(\infty\).  The standard lift an \(\mathcal{O}_2\)-rational point on the boundary is given by:

\[ P = \begin{bmatrix}
    \frac{-|z|^2+iv}{2} \\
    z \\
    1 \\
   \end{bmatrix}\]
where \(z\) and \(v\) come from the horospherical coordinates of the point \(\pi(\psi(z,v,0))\).  The standard lift may not be integral, but we can multiply by some \(q \in \mathcal{O}_2\) so that \(qP\) is then a primitive integral lift.  To calulate the level of \(p\), we compute
\[ |\langle P_{\infty}, qP \rangle|^{2} = |(qP)^{*}JP_{\infty}|^2 = |q|^2 \]  In order to find the \(\mathcal{O}_2\)-rational points of depth \(k\) for \(1 \leq k \leq 16\), we need to find \(q \in \mathcal{O}_2\) such that \(|q|^2 = k\) for \(1 \leq k \leq 16\).  We also need only find values of \(q\) up to multiplication by a unit.  Each value of \(q\) is of the form \(q = a + bi\sqrt{2}\) for \(a,b \in \Z\).  \(|q|^2 = a^2 +2b^2\), so we need to find values \(a,b \in \Z\) such that \(a^2+2b^2 = k\) for \(1 \leq k \leq 16\).  After some simple calculations, we have the following:
\small
\begin{center}
 \begin{tabular}{|c|c|}
  \hline
  \(\textrm{Depth}\) & Possible \(q\)'s \\
  \hline
  1 & 1 \\
  \hline
  2 & \(i\sqrt{2}\) \\
  \hline
  3 & \(1+i\sqrt{2},1-i\sqrt{2}\) \\
  \hline
  4 & 2 \\
  \hline
  6 & \(2+i\sqrt{2}, 2 - i\sqrt{2}\) \\
  \hline
  8 & \(2i\sqrt{2}\)\\
  \hline
  9 & \(3,1+2i\sqrt{2},1-2i\sqrt{2}\) \\
  \hline
  11 & \(3+i\sqrt{2},3-i\sqrt{2}\) \\
  \hline
  12 & \(2+2i\sqrt{2}, 2-2i\sqrt{2}\)\\
  \hline
  16 & 4\\
  \hline
 \end{tabular}
\end{center}
 \normalsize
 Next, we determine which horospherical points, \((z,v)\) have standard lifts \(P\), such that \(qP\) is a primitive integral lift.  In order to accomplish this task, we find all \((z,v) \in D_{\infty}(2)\) such that for a fixed \(q\), \(qz \in \mathcal{O}_2\) and \(qP_1 \in \mathcal{O}_2\) where \(P_1\) is the first entry of the standard lift of \(P\).  We only list depths that contain \(\mathcal{O}_2\)-rational points.  Moreover, we need only consider one representative from each \(\Gamma_{\infty}(2)\)-orbit. We list the set of \(\Gamma_{\infty}(2)\)-orbit representatives up to depth 16 in Appendix \ref{2_points}.
 \\[1\baselineskip]
Now for each representative, \(p\), we need to find a matrix \(A_p \in \Gamma(2)\) such that
\(A_p(\infty) = p\).

\subsection{Matrices and primitive integral lifts for \(\mathcal{O}_{2}\)-rational representatives}
 For each \(\mathcal{O}_2\)-rational point \(p\), there exists a \(q \in \mathcal{O}_2\) such that when we scale the standard lift of \(p\) by \(q\), we get a primitive integral lift of \(p\).  This practice takes care of finding primitive integral lifts for all of our \(\mathcal{O}_2\)-rational points of interest.  We look for a matrix that sends \(\infty = \pi(1,0,0)^{T}\) to \(p\).  In other words, we need to find a unitary matrix that has the primitive integral lift of \(p\) as its first column.  (First column is image of basis vector \((1,0,0)^{T}\))  There does not seem to be a general procedure for finding these matrices sending \(\infty\) to our points of interest.  A couple ``tricks'' one can use are:
 \\[1\baselineskip]
1. Use stabilizers of the vertical complex line in the Heisenberg group.  We use this trick to find matrices sending \(\infty\) to points with horospherical coordinate \(z=0\). \newline
2. Hit all relevant integral points by group elements we have already found, and see if we land in the \(\Gamma_{\infty}(2)\)-orbit of the point we are trying to reach.\newline
3. From a previously found matrix, use elementary row/column operations, transposition, complex conjugation, inversion, to see if we can get a matrix whose first column is the primitive integral lift of interest.
\\[1\baselineskip]
We denote \(p_{i,j}\) the \(j\)th \(\mathcal{O}_2\)-rational point of depth \(i\), and \(A_{i,j}\) a matrix sending \(\infty\)  We denote \(p_0=(0,0)\) and \(I_0\) a matrix sending \(\infty\) to \((0,0)\).  We list the matrices sending \(\infty\) to \(p_{i,j}\) in Appendix \ref{2_matrices}, and we only list matrices, as the primitive integral lifts corresponding to \(p_{i,j}\) are simply the first column of \(A_{i,j}\).
\\[1\baselineskip]
By Proposition \ref{gener}, we have that \(\Gamma(2)\) is generated by:
\[S(2) = \{T_2,T_{i\sqrt{2}},T_{2\sqrt{2}},R\} \cup \mathcal{A}\]
where \(T_2,T_{i\sqrt{2}},T_{2\sqrt{2}},R\) are the generators of \(\Gamma_{\infty}(2)\) and \(\mathcal{A}\) is the set of matrices listed in Appendix \ref{2_matrices}.  Using Lemma \ref{expbds2}, and the aid of MATLAB for computations \cite{10}, we can cycle through all possible elements of \(\Gamma_{\infty}(2)\) that could appear in a relation for \(\Gamma(2)\).  Using the MAGMA function ``Simplify(\(\cdot\))'' \cite{7} for simplifying the set of generators and the set of relations, we obtain the presentation in Theorem 2.  The MAGMA computation reduces our original generating set of 54 matrices to a generating set of 3 matrices and our original set of 5,837 relations to a set of 29 relations.  The abelianization in Corollary \ref{ab2} of \(\Gamma(2)\) is obtained using the MAGMA function ``AbelianQuotient(\(\cdot\))'' \cite{7}.

\section{A presentation for \(\PU(2,1;\mathcal{O}_{11})\)}
\noindent For the entirety of this section, we denote \(\tau = \frac{1+i\sqrt{11}}{2}\).  We omit some details in this case, as the proofs follow identical procedures to that of \(d=2\).
\begin{thm}\label{d=11 pres}
  \(\Gamma(11) = \PU(2,1;\mathcal{O}_{11})\) admits the following presentation:
  \small
  \begin{align*}
 &\Gamma(11) = \langle R,T_{1},T_{v},I,A :R^2,I^2,A^4,[R,T_{v}],[T_{1},T_{v}],(IR)^2,RT_{v}T_{1}^{-1}RT_{1}^{-1},T_{1}^{-1}IT_{1}IA^2I,\\
&IT_{1}^{-1}AT_{1}^{-1}AIT_{v}^{-1},(RAIA^{-1})^2,A^{-1}T_{1}IT_{1}^{-1}IT_{1}^{-1}A^{-1}IT_{1}^{-1},A^{-1}T_{v}RT_{1}^{-1}IT_{v}IRA^{-1}T_{1},\\
&A^{-1}RT_{v}IA^{-2}RT_{1}IAT_{1},A^{-1}T_{v}RIA^{-2}RT_{1}IAT_{1},(T_{1}IT_{1}^{-1}A^{-1}RA)^2,(T_{v}^{-1}AIT_{1}IT_{1}^{-1}A^{-1}I)^2,\\
&IRT_{v}^{-1}IT_{v}^{-1}T_{1}RAT_{1}IT_{1}^{-1}AT_{1}IAT_{1}IT_{1}^{-1},IT_{1}^{-1}A^{-1}IRT_{v}^{-1}T_{1}IT_{v}RT_{1}^3IT_{1}^{-1}AIRT_{1}^{-2}R,\\
&T_{v}A^{-1}T_{1}IT_{1}^{-1}A^{-1}T_{1}IT_{v}IA^{-2}RAIA^{-1}T_{1}IT_{1}^{-1}RT_{1}IT_{1}^{-2},\\
&AT_{1}IT_{1}^{-1}A^2T_{1}AT_{1}IT_{1}AT_{1}IA^{-1}T_{v}RT_{1}^{-1}IAIA^{-1}RIT_{1}^{-1}AT_{1}IT_{1}^{-1}A,\\
&IA^{-2}T_{1}AT_{1}IT_{1}AT_{1}IT_{1}^{-1}A^{-1}RA^{-1}T_{1}IT_{1}^{-1}A^{-1}IT_{v}^{-1}T_{1}AIA^{-1}T_{1}IT_{1}^{-1}RA,\\
&AT_{1}IT_{1}^{-1}A^{-2}T_{1}IT_{1}^{-1}A^2T_{1}IT_{1}^{-1}A^2T_{1}IT_{1}^{-1}A^2T_{1}IT_{1}^{-1}A^2T_{1}IT_{1}^{-1}A,\\
&AIA^{-1}RT_{v}IA^{-1}T_{1}AIRT_{1}^{-1}A^{-1}T_{1}AT_{1}IT_{v}RIT_{1}^{-1}T_{v}IAT_{1}^{-1}ARA^{-2}IT_{v}^{-1},\\
&AT_{1}^2IT_{1}^{-1}A^{-2}T_{1}IT_{1}^{-1}A^{-1}RAT_{1}IA^{-2}T_{1}IT_{1}^{-1}A^{-2}T_{1}IT_{1}^{-1}A^{-2}T_{1}^{-1}A^{-1}RT_{1}I,\\
&A^{-2}T_{1}IT_{1}^{-1}A^2T_{1}IT_{1}^{-1}A^{-1}T_{1}RIAT_{1}IT_{1}^{-1}A^2T_{1}IT_{1}^{-1}A^{-1}IT_{v}^{-1}T_{1}RAT_{1}IT_{1}^{-1},\\
&T_{1}^{-1}IAT_{1}IT_{1}^{-1}AT_{v}RIA^{-1}T_{1}IA^{-1}RIT_{1}^{-1}AIT_{1}^{-1}A^{-1}IT_{1}^{-1}AT_{1}IT_{1}^{-1}ARAT_{1}IA^{-1}T_{v}R\rangle\\
\end{align*}

\normalsize
The unitary matrices corresponding to the generators of \(\Gamma(11)\) are given by:
\[ R = \begin{bmatrix}
                        1 & 0 & 0 \\
                        0 & -1 & 0 \\
                        0 & 0 & 1 \\
                       \end{bmatrix},
T_1 = \begin{bmatrix}
         1 & -1 & -1 + \tau \\
         0 & 1 & 1 \\
         0 & 0 & 1 \\
        \end{bmatrix},
     T_v = \begin{bmatrix}
         1 & 0 & -1 + 2\tau \\
         0 & 1 & 0 \\
         0 & 0 & 1 \\
        \end{bmatrix},   
I=I_0 = \begin{bmatrix}
         0 & 0 & 1 \\
         0 & -1 & 0 \\
         1 & 0 & 0 \\
        \end{bmatrix},\]
\[       A = A_{5,1} = \begin{bmatrix}
        -1 - \tau & -2 & -1 + \tau \\
        -1 & -1 + \tau & 1 \\
        -2+\tau & \tau & 1 \\
        \end{bmatrix}
\]

\end{thm}
\begin{cor}\label{ab11}
 The abelianization of \(\Gamma(11)\) is \(\Z/2\Z \times \Z/2\Z \times \Z/2\Z\).
\end{cor}
\subsection{Finding \(D_{\infty}(11)\) and \(\Gamma_{\infty}(11)\)}
  \noindent Once again, we let \(B((z,v),r)\) be the open extended Cygan ball centered at \(p = (z,v) \in \partial_{\infty}\mathbb{H}^2_{\C}\) with radius \(r\).  Recall that balls of depth \(n\) appear at height \(u(n) =\frac{2}{\sqrt{n}}\).  The covering argument for \(d=11\) follows the same line of reasoning as \(d=2\).  More information on the cusp stabilizer, \(\Gamma_{\infty}(11)\) can be found in Section 5.3 of \cite{9}.
  \begin{lemma}[\cite{9}]
   1. The cusp stabilizer, \(\Gamma_{\infty}(11)\), admits the following presentation:
\[\Gamma_{\infty}(11) = \left\langle T_{1},T_{\tau},T_{v},R:T_{v}[T_{1},T_{\tau}],[T_{1},T_{v}],[T_{\tau},T_{v}],R^{2},[R,T_{v}],[T_1,R]T_1^{-2}T_v,[T_{\tau},R]T_{\tau}^{-2}T_v \right\rangle\]
2. Let \(D_{\infty}(11) \subset \partial_{\infty} \mathbb{H}^{2}_{\C}\) be the affine convex hull of the points with horospherical coordinates \((0,0),(1,0),(\tau,0),\)
\((0,2\sqrt{11}),(1,2\sqrt{11}),(\tau,2\sqrt{11})\).  Then \(D_{\infty}(11)\) is a fundamental domain for \(\Gamma_{\infty}(11)\) acting on \(\partial_{\infty} \mathbb{H}^{2}_{\C} - \{\infty\}\).
\\[1\baselineskip]
We use the following generators for \(\Gamma_{\infty}(11)\):

\[T_1 = \begin{bmatrix}        
         1 & -1 & -1+\tau \\
         0 & 1 & 1 \\
         0 & 0 & 1 \\
        \end{bmatrix}\textrm{, }
        T_{\tau} =  \begin{bmatrix}        
         1 & -1+\tau & -2+\tau \\
         0 & 1 & \tau \\
         0 & 0 & 1 \\
        \end{bmatrix}\textrm{, }
        T_{v} = \begin{bmatrix}
                        1 & 0 & 1-2\tau \\
                        0 & 1 & 0 \\
                        0 & 0 & 1 \\
                       \end{bmatrix}\textrm{, }
                       R = \begin{bmatrix}
                            1 & 0 & 0 \\
                            0 & -1 & 0 \\
                            0 & 0 & 1 \\
                           \end{bmatrix}\]
where \(T_{z}\) is a Heisenberg translation matrix (eq. \ref{trans}) and \(R\) a Heisenberg rotation matrix (eq. \ref{rot}).
  \end{lemma}
\noindent Following a similar line of reasoning as in the \(d=2\) case, we have:
\begin{lemma}\label{nmfm11}
 For any \(\gamma \in \Gamma_{\infty}(11)\), \(\gamma\) can be written as:
 \[ \gamma = R^{p}{T_{v}}^{n}{T_1}^{m}{T_{\tau}}^{l} \textrm{ for some } m,n,l \in \Z
 \textrm{ and } p = 0,1\]
\end{lemma}

\begin{lemma}\label{expbds11}
If \(\Gamma(11)\) admits the relation \(A^{-1}_c(\gamma^{3}_{\infty})^{-1}A_a\gamma^{1}_{\infty}A_b = \gamma^{*}_{\infty}\) for \(a,b,c \in \{1,...,r\}\) and \newline \(\gamma^1_{\infty}\),\(\gamma^3_{\infty}\),\(\gamma^{*}_{\infty} \in \Gamma_{\infty}(11)\), then, using the normal form as in Lemma \ref{nmfm11}, the exponents of \(\gamma^1_{\infty}\), \(\gamma^3_{\infty}\), and \(\gamma^{*}_{\infty}\) satisfy 
 \[|n|\leq 21 \textrm{, } |m|\leq 9 \textrm{, } |l|\leq 5\]
\end{lemma}

\noindent In order to obtain the bound above, one can easily check the bounds from Lemma \ref{cybounds} for the \(d=11\) case satisfy:
\[\left(\frac{4dep(p_{c})}{dep(p_{a})dep(p_{b})}\right)^{\frac{1}{4}} < 3.4880 \textrm{, } 
d_{C}(p_{b},p_{0}) < 2.5661 \textrm{, }
d_{C}(p_{0},{A_{a}}^{-1}(\infty)) < 2.6901\]

\subsection{Covering depth of \(\Gamma(11)\)}

\begin{lemma}
 Let \(u = u(44) + \epsilon =0.3015114\), and \(H_u\) be the horosphere of height \(u\) based at \(\infty\).  Then the prism, \(D_{\infty}(11) \times \{u\}\), is covered by the intersections with \(H_u\) of the following extended Cygan balls:
 \small
 \\[1\baselineskip]
 Depth 1: \(B((0,0),\sqrt{2}), B((0,2\sqrt{11}),\sqrt{2}), B((1,\sqrt{11}),\sqrt{2}), B((\tau,\sqrt{11}),\sqrt{2}),B((-1+\tau,\sqrt{11}),\sqrt{2}),B((1+\tau,\sqrt{11}),\sqrt{2})\)
 \\[1\baselineskip]
 Depth 3: \(B((\frac{1}{3}\tau,\frac{5}{3}\sqrt{11}),(\frac{4}{3})^{\frac{1}{4}}), B((\frac{1}{3} + \frac{2}{3}\tau,\frac{5}{3}\sqrt{11}),(\frac{4}{3})^{\frac{1}{4}}),B((\frac{2}{3} + \frac{1}{3}\tau,\sqrt{11}),(\frac{4}{3})^{\frac{1}{4}}) \)
 \\[1\baselineskip]
 Depth 4: \(B((0,\sqrt{11}),1), B((1,2\sqrt{11}),1)\)
 \\[1\baselineskip]
 Depth 5: \(B((\frac{1}{5}+\frac{2}{5}\tau,\frac{1}{5}\sqrt{11}),(\frac{4}{5})^\frac{1}{4})\) \indent \indent \indent \indent \indent \indent Depth 23: \(B((\frac{3}{23}+\frac{4}{23}\tau,\frac{31}{23}\sqrt{11}),(\frac{4}{23})^\frac{1}{4})\)
 \\[1\baselineskip]
 Depth 9: \(B((0,\frac{4}{3}\sqrt{11}),(\frac{4}{9})^\frac{1}{4}), B((1,\frac{5}{3}\sqrt{11}),(\frac{4}{9})^\frac{1}{4})\)
  \indent \indent Depth 25: \(B((\frac{14}{25}+\frac{4}{25}\tau,\frac{36}{25}\sqrt{11}),(\frac{4}{25})^\frac{1}{4})\)
 \normalsize
\end{lemma}

\begin{cor}
 The covering depth of \(\Gamma(11)\) is at most \(43\).
\end{cor}
\begin{figure}\label{fig3}
 \caption{Covering of \(D_{\infty}(11) \times \{u\}\) with Cygan balls up to depth 43}
 \includegraphics[scale = 0.2]{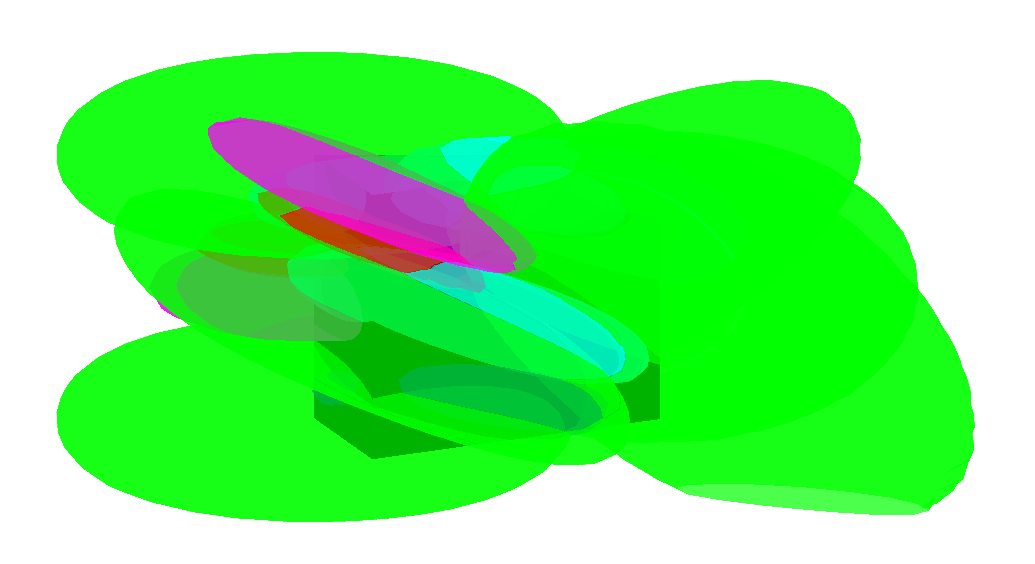}
 \centering
\end{figure}
\noindent Although we only used Cygan balls of depth up to 25, it appears that we still need to pass to depth 43.  We generated pictures (Figure 3) of coverings of \(D_{\infty}(11)\) by Cygan balls.  For heights corresponding to depths \(n \leq 42\), it appeared that balls of depth at most \(n\) did not cover the prism.  A proof similar to the \(d=2\) case makes rigorous the fact that passing to depth \(43\) is sufficient in covering \(\mathbb{H}_{\C}^2\).

\begin{figure}[h]\label{fig4}
 \caption{An affine decomposition of \(D_{\infty}(11) \times \{u\}\)}
 \includegraphics[scale = 0.12]{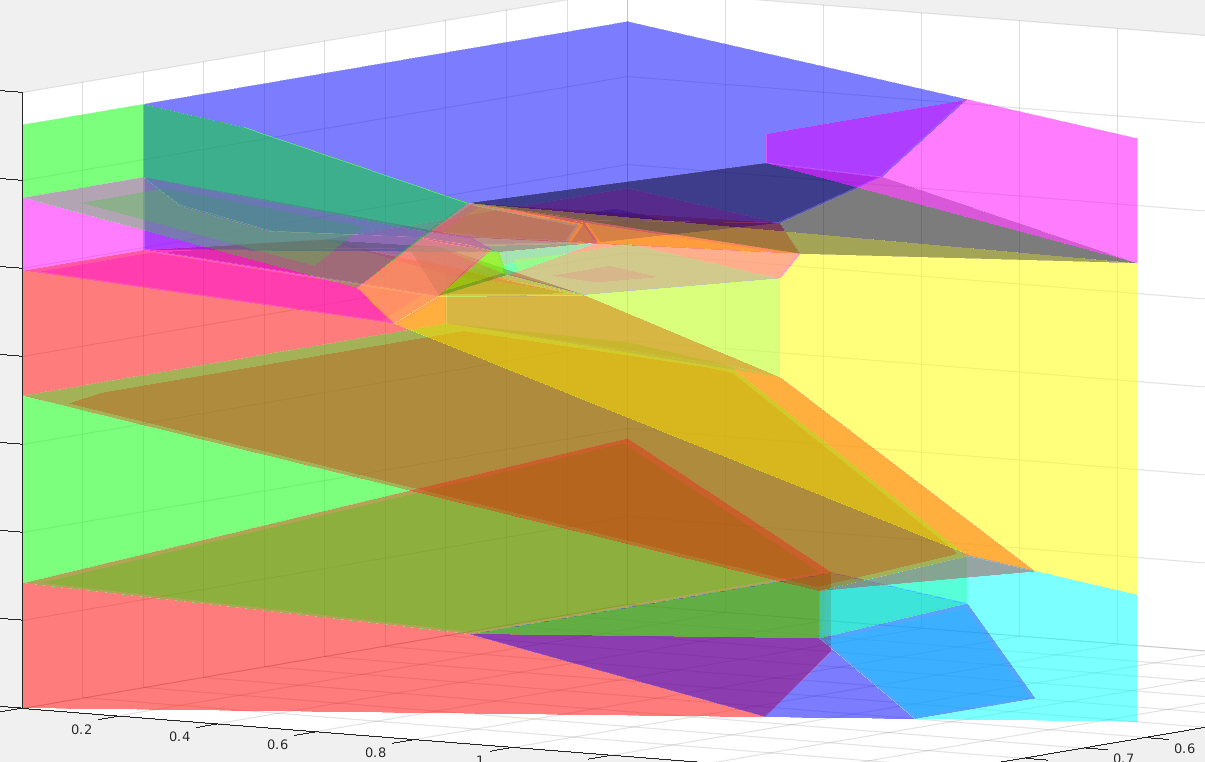}
 \centering
\end{figure}

\subsection{\(\mathcal{O}_{11}\)-rational points of depth at most \(n=43\)}
To find all the points of depth at most 43 in \(D_{\infty}(11)\), we perform the same type of calculations as the \(d=2\) case. Below are the possible values of ``\(q\)'', where \(q\) is the scale factor from Section 5.4.  We list only the depths where a \(q\) factor is possible.  
\small
\begin{center}
 \begin{tabular}{|c|c|}
  \hline
  \(\textrm{Depth}\) & Possible \(q\)'s \\
  \hline
  1 & 1 \\
  \hline
  3 & \(-1+\tau,\tau\) \\
  \hline
  4 & 2 \\
  \hline
  5 & \(-2+\tau,1+\tau\) \\
  \hline
  9 & \(3,-3+\tau,2+\tau\) \\
  \hline
  11 & \(-1+2\tau\) \\
  \hline
  12 & \(-2+2\tau, 2\tau\)\\
  \hline
  15 & \(-4+\tau,-3+2\tau,1+2\tau,3+\tau\) \\
  \hline
  16 & 4\\
  \hline
  20 & \(-4+2\tau,2+2\tau\) \\
  \hline
  23 & \(-5+\tau,4+\tau\) \\
  \hline
  25 & \(5, -2+3\tau, -1+3\tau\) \\
  \hline
  27 & \(-5+2\tau,-3+3\tau,3+2\tau,3\tau\) \\
  \hline
  31 & \(-4+3\tau,1+3\tau\) \\
  \hline
  33 & \(-6+\tau,5+\tau\) \\
  \hline
  36 & \(6,-6+2\tau, 4+2\tau\)\\
  \hline
  37 & \(-5+3\tau,2+3\tau\) \\
  \hline
 \end{tabular}
\end{center}
\normalsize
We obtain a large number of \(\Gamma_{\infty}(11)\)-orbit representatives in the \(d=11\) case.  For this reason, we list the points in a separate document, ``rational\_points\_matrices\_d11.pdf'', which can be found in \cite{10}.

\subsection{Matrices and primitive integral lifts for \(\mathcal{O}_{11}\)-rational representatives}
  Due to the large number of \(\Gamma_{\infty}\)-orbit representatives of \(\mathcal{O}_{11}\)-rational points in \(D_{\infty}(11)\), our generating set for \(\Gamma(11)\) is substantial, 259 matrices to be exact.  We list these generators in ````rational\_points\_matrices\_d11.pdf'', \cite{10}.  Once again, we do not list primitive integral lifts of the relevant points, as one can recover the primitive integral lifts of these points from the first column of each matrix respectively.  By Proposition \ref{gener}, we have that \(\Gamma(11)\) is generated by:
\[S(11) = \{T_1,T_{\tau},T_{v},R\} \cup \mathcal{A}\]
where \(T_1,T_{\tau},T_{v},R\) are the generators of \(\Gamma_{\infty}(11)\) and \(\mathcal{A}\) is the set of matrices listed in \cite{10}. Using Lemma \ref{expbds11}, and the aid of MATLAB for computations \cite{10}, we can cycle through all possible elements of \(\Gamma_{\infty}(11)\) that could appear in a relation for \(\Gamma(11)\).  Using the MAGMA function ``Simplify(\(\cdot\))'' \cite{7} for simplifying the set of generators and the set of relations, we obtain the presentation in Theorem 3.  The MAGMA computation reduces our original generating set of 263 matrices to a generating set of 5 matrices and our original set of 23,673 relations to a set of 26 relations.  Note, \(\Gamma(11)\) has a presentation involving only 3 generators, but the number of relations and length of some relations is much larger using this generating set.  When using MAGMA, a user has some control over the simplification of a particular presentation.  One can retain particular generators using the parameter ``Preserve:=[]'' in the MAGMA function ``Simplify(\(\cdot\))'', and control the elmination of relations in ''Simplify(\(\cdot\))'' using the parameter ``EliminationLimit:='' \cite{7}.  Once again, we use the function ``AbelianQuotient(\(\cdot\))'' in MAGMA to obtain the abelianization of \(\Gamma(11)\) in Corollary \ref{ab11} \cite{7}.
  
\newpage
\appendixtitleon
\begin{appendices}
\section{\(\mathcal{O}_2\)-rational point representatives}\label{2_points}
\small
\begin{center}
 \begin{tabular}{|c|c|}
  \hline
  Depth & \(\mathcal{O}_2\)-Rational Points \\
  \hline
  1 & \((0,0)\) \\
  \hline
  2 & \((0,\sqrt{2})\) \\
  \hline
  3 & \((\frac{2}{3}+\frac{i\sqrt{2}}{3},\frac{2}{3}\sqrt{2}),(\frac{2}{3}+\frac{2i\sqrt{2}}{3},\frac{2}{3}\sqrt{2}), (\frac{4}{3}+\frac{i\sqrt{2}}{3},0)\) \\
  \hline
  4 & \((1,0),(1,\sqrt{2})\) \\
  \hline
  6 & \((\frac{2}{3}+\frac{2i\sqrt{2}}{3},\frac{5}{3}\sqrt{2}),(\frac{4}{3}+\frac{i\sqrt{2}}{3},\sqrt{2}),(\frac{2}{3}+\frac{i\sqrt{2}}{3},\frac{5}{3}\sqrt{2})\) \\
  \hline
  8 & \((0,\frac{\sqrt{2}}{2}),(0,\frac{3}{2}\sqrt{2}),(1,\frac{\sqrt{2}}{2}),(1,\frac{3}{2}\sqrt{2})\) \\
  \hline
  9 &  \((0,\frac{2}{3}\sqrt{2}),
  (0,\frac{4}{3}\sqrt{2}),
  (\frac{2}{3}+\frac{i\sqrt{2}}{3},0),
  (\frac{2}{3}+\frac{i\sqrt{2}}{3},\frac{4}{3}\sqrt{2}),
  (\frac{2}{3}+\frac{2i\sqrt{2}}{3},0),\)\\
  & \((\frac{2}{3}+\frac{2i\sqrt{2}}{3},\frac{4}{3}\sqrt{2}),
  (\frac{4}{3}+\frac{i\sqrt{2}}{3},\frac{2}{3}\sqrt{2}),
  (\frac{4}{3}+\frac{i\sqrt{2}}{3},\frac{4}{3}\sqrt{2}),\)\\
  & \((\frac{2}{9}+\frac{5i\sqrt{2}}{9},\frac{4}{3}\sqrt{2}),
  (\frac{4}{9}+\frac{i\sqrt{2}}{9},\frac{4}{9}\sqrt{2}),
  (\frac{8}{9}+\frac{2i\sqrt{2}}{9},\frac{16}{9}\sqrt{2}),
  (\frac{2}{9}+\frac{4i\sqrt{2}}{9},\frac{10}{9}\sqrt{2}),
  (\frac{10}{9}+\frac{2i\sqrt{2}}{9},\frac{4}{3}\sqrt{2}),
  (\frac{14}{9}+\frac{i\sqrt{2}}{9},\frac{10}{9}\sqrt{2})\)\\
  \hline
  11 & \((\frac{2}{11}+\frac{3i\sqrt{2}}{11},\frac{8}{11}\sqrt{2}),
  (\frac{4}{11}+\frac{6i\sqrt{2}}{11},\frac{10}{11}\sqrt{2}),
  (\frac{8}{11}+\frac{i\sqrt{2}}{11},\frac{2}{11}\sqrt{2}),
  (\frac{10}{11}+\frac{4i\sqrt{2}}{11},\frac{4}{11}\sqrt{2}),
  (\frac{16}{11}+\frac{2i\sqrt{2}}{11},\frac{8}{11}\sqrt{2}),\)\\
  & \((\frac{2}{11}+\frac{8i\sqrt{2}}{11},\frac{18}{11}\sqrt{2}),
  (\frac{4}{11}+\frac{5i\sqrt{2}}{11},\frac{20}{11}\sqrt{2}),
  (\frac{6}{11}+\frac{2i\sqrt{2}}{11},\frac{6}{11}\sqrt{2}),
  (\frac{12}{11}+\frac{4i\sqrt{2}}{11},\frac{2}{11}\sqrt{2}),
  (\frac{14}{11}+\frac{i\sqrt{2}}{11},\frac{16}{11}\sqrt{2})\)\\
  \hline
  12 &\((\frac{1}{3}+\frac{2i\sqrt{2}}{3},0),
  (\frac{1}{3}+\frac{2i\sqrt{2}}{3},\sqrt{2}),
  (\frac{1}{3}+\frac{2i\sqrt{2}}{3},\frac{16}{9}\sqrt{2}),
  (\frac{2}{9}+\frac{4i\sqrt{2}}{9},2\sqrt{2}),
  (\frac{1}{3}+\frac{i\sqrt{2}}{3},\frac{2}{3}\sqrt{2}),
  (\frac{1}{3}+\frac{i\sqrt{2}}{3},\frac{5}{3}\sqrt{2})\)\\
  \hline
  16 &\((\frac{i\sqrt{2}}{2},0),
  (\frac{i\sqrt{2}}{2},\frac{\sqrt{2}}{2}),
  (\frac{i\sqrt{2}}{2},\sqrt{2}),
  (\frac{i\sqrt{2}}{2},\frac{3}{2}\sqrt{2}),
  (1+\frac{i\sqrt{2}}{2},0),
  (1+\frac{i\sqrt{2}}{2},\frac{\sqrt{2}}{2}),
  (1+\frac{i\sqrt{2}}{2},\sqrt{2}),\) \\
   & \((1+\frac{i\sqrt{2}}{2},\frac{3}{2}\sqrt{2})\)\\
  \hline
 \end{tabular}

\end{center}
\normalsize

\begin{section}{Matrices sending \(\infty\) to \(\mathcal{O}_{2}\)-rational points}\label{2_matrices}
 \scriptsize
\[I_0 = \begin{bmatrix}
         0 & 0 & 1 \\
         0 & -1 & 0 \\
         1 & 0 & 0 \\
        \end{bmatrix},
   A_{2,1} =\begin{bmatrix}
         -1 & 0 & i\sqrt{2} \\
         0 & 1 & 0  \\
         i\sqrt{2} & 0 & 1 \\
        \end{bmatrix},
 A_{3,1} = \begin{bmatrix}
         -1 & 0 & 0 \\
         i\sqrt{2} & -1 & 0 \\
         1+i\sqrt{2} & i\sqrt{2} & -1 \\
        \end{bmatrix},
   A_{3,2} =\begin{bmatrix}
         i\sqrt{2} & 0 & 1 \\
         2 & -1 & -i\sqrt{2} \\
         1-i\sqrt{2} & i\sqrt{2} & -1 \\
        \end{bmatrix},\]
\[A_{3,3} = \begin{bmatrix}
         -1+i\sqrt{2} & 2i\sqrt{2} & 2-i\sqrt{2} \\
         2-i\sqrt{2} & 1-2i\sqrt{2} & -2 \\
         1-i\sqrt{2} & -i\sqrt{2} & -1 \\
        \end{bmatrix},
    A_{4,1} = \begin{bmatrix}
         -1 & -2 & 2 \\
         2 & 3 & -2 \\
         2 & 2 & -1 \\
        \end{bmatrix},
   A_{4,2} =\begin{bmatrix}
         -1+i\sqrt{2} & -2+i\sqrt{2} & 3 \\
         2 & 3 & -2-i\sqrt{2} \\
         2 & 2 & -1-i\sqrt{2} \\
        \end{bmatrix},\]
     \[ A_{6,1} = \begin{bmatrix}
         -3+i\sqrt{2} & 2-i\sqrt{2} & 3+3i\sqrt{2} \\
         2i\sqrt{2} & -1-2i\sqrt{2} & 4-i\sqrt{2} \\
         2+i\sqrt{2} & -2 & 1-2i\sqrt{2} \\
        \end{bmatrix},
    A_{6,2} =\begin{bmatrix}
         -3 & -2+2i\sqrt{2} & 2+2i\sqrt{2} \\
         2+2i\sqrt{2} & 3 & 2-2i\sqrt{2} \\
         2+i\sqrt{2} & 2 & 1-2i\sqrt{2} \\
        \end{bmatrix},\]
\[ A_{6,3} = \begin{bmatrix}
         1+2i\sqrt{2} & 2+2i\sqrt{2} & 2-2i\sqrt{2} \\
         2 & 3 & -2-2i\sqrt{2}\\
         2-i\sqrt{2} & 2-2i\sqrt{2} & -3 \\
        \end{bmatrix},
A_{8,1} = \begin{bmatrix}
         -1 & 0 & i\sqrt{2} \\
         0 & 1 & 0 \\
         2i\sqrt{2} & 0 & 3 \\
        \end{bmatrix},
A_{8,2} =\begin{bmatrix}
         -3 & 0 & i\sqrt{2} \\
         0 & 1 & 0 \\
         2i\sqrt{2} & 0 & 1 \\
        \end{bmatrix},\]
\[  A_{8,3} = \begin{bmatrix}
         -1-i\sqrt{2} & -2-i\sqrt{2} & 1+i\sqrt{2} \\
         2i\sqrt{2} & 1+2i\sqrt{2} & -i\sqrt{2} \\
         2i\sqrt{2} & 2i\sqrt{2} & 1-2i\sqrt{2} \\
        \end{bmatrix},
   A_{8,4} =\begin{bmatrix}
         -3-i\sqrt{2} & -4-i\sqrt{2} & 3\\
         2i\sqrt{2} & 1+2i\sqrt{2} & -2-i\sqrt{2} \\
         2i\sqrt{2} & 2i\sqrt{2} & -1-i\sqrt{2} \\
        \end{bmatrix},\]
\[A_{9,1} = \begin{bmatrix}
         i\sqrt{2} & 0 & 1 \\
         0 & 1 & 0 \\
         3 & 0 & -i\sqrt{2}\\
        \end{bmatrix},
A_{9,2} = \begin{bmatrix}
         2i\sqrt{2} & 0 & -1 \\
         0 & 1 & 0 \\
         3 & 0 & i\sqrt{2} \\
        \end{bmatrix},
   A_{9,3} =\begin{bmatrix}
         -1 & 0 & 0 \\
         2+i\sqrt{2} & 1 & 0 \\
         3 & 2-i\sqrt{2} & -1 \\
        \end{bmatrix}, \]
\[A_{9,4} = \begin{bmatrix}
         -1+2i\sqrt{2} & 2i\sqrt{2} & -i\sqrt{2} \\
         2+i\sqrt{2} & 3 & -2 \\
         3 & 2-i\sqrt{2} & -1+i\sqrt{2} \\
        \end{bmatrix},
   A_{9,5} =\begin{bmatrix}
         -2 & -2+2i\sqrt{2} & 3-2i\sqrt{2} \\
         2+2i\sqrt{2} & 5 & -6 \\
         3 & 2-2i\sqrt{2} & -2+3i\sqrt{2} \\
        \end{bmatrix},\]
\[A_{9,6} = \begin{bmatrix}
         -2+2i\sqrt{2} & -2-2i\sqrt{2} & -1-2i\sqrt{2} \\
         2+2i\sqrt{2} & -3 & -2 \\
         3 & -2+2i\sqrt{2} & -2+i\sqrt{2} \\
        \end{bmatrix},
A_{9,7} = \begin{bmatrix}
         -3+i\sqrt{2} & 2-2i\sqrt{2} & 2 \\
         4+i\sqrt{2} & -5 & -2-2i\sqrt{2} \\
         3 & -4+i\sqrt{2} & -3-i\sqrt{2} \\
        \end{bmatrix},\]
\[A_{9,8} =\begin{bmatrix}
         -3+2i\sqrt{2} & -4+4i\sqrt{2} & 4-3i\sqrt{2} \\
         4+i\sqrt{2} & 7 & -6 \\
         3 & 4-i\sqrt{2} & -3+i\sqrt{2} \\
        \end{bmatrix},
      A_{9,9} = \begin{bmatrix}
         -3 & -2-i\sqrt{2} & 1+i\sqrt{2} \\
         -2+i\sqrt{2} & -3 & 2 \\
         1+2i\sqrt{2} & 2i\sqrt{2} & -i\sqrt{2} \\
        \end{bmatrix},\]
\[A_{9,10} =\begin{bmatrix}
         -1 & 0 & i\sqrt{2} \\
         i\sqrt{2} & -1 & 2 \\
         1+2i\sqrt{2} & i\sqrt{2} & 3-i\sqrt{2} \\
        \end{bmatrix},
        A_{9,11} = \begin{bmatrix}
         -4 & -2i\sqrt{2} & 1+2i\sqrt{2} \\
         2i\sqrt{2} & -1 & 2 \\
         1+2i\sqrt{2} & -2 & 2-i\sqrt{2} \\
        \end{bmatrix},\]
\[ A_{9,11} = \begin{bmatrix}
         -4 & -2i\sqrt{2} & 1+2i\sqrt{2} \\
         2i\sqrt{2} & -1 & 2 \\
         1+2i\sqrt{2} & -2 & 2-i\sqrt{2} \\
        \end{bmatrix},
        A_{9,12} = \begin{bmatrix}
         2+i\sqrt{2} & -2 & 1-2i\sqrt{2} \\
         2 & -1 & -2i\sqrt{2} \\
         1-2i\sqrt{2} & 2i\sqrt{2} & -4 \\
        \end{bmatrix},\]
\[A_{9,13} =\begin{bmatrix}
         2+2i\sqrt{2} & -2+2i\sqrt{2} & -3 \\
         2-2i\sqrt{2} & 3 & 2+2i\sqrt{2} \\
         1-2i\sqrt{2} & 2 & 2+i\sqrt{2} \\
        \end{bmatrix},
      A_{9,14} = \begin{bmatrix}
         1+3i\sqrt{2} & -2i\sqrt{2} & -4 \\
         2-3i\sqrt{2} & -1+2i\sqrt{2} & 4+2i\sqrt{2} \\
         1-2i\sqrt{2} & i\sqrt{2} & 3+i\sqrt{2} \\
        \end{bmatrix},\]
\[A_{11,1} =\begin{bmatrix}
         -1+i\sqrt{2} & -i\sqrt{2} & 1 \\
         i\sqrt{2} & -1 & 0 \\
         3+i\sqrt{2} & -2 & -i\sqrt{2} \\
        \end{bmatrix},
        A_{11,2} = \begin{bmatrix}
         -2+i\sqrt{2} & 2 & 1 \\
         2i\sqrt{2} & 1-2i\sqrt{2} & 2-i\sqrt{2} \\
         3+i\sqrt{2} & -2-i\sqrt{2} & -1 \\
        \end{bmatrix},\]
\[A_{11,3} = \begin{bmatrix}
         -1 & 0 & 0 \\
         2+i\sqrt{2} & -1 & 0 \\
         3+i\sqrt{2} & -2+i\sqrt{2} & -1 \\
        \end{bmatrix},
   A_{11,4} =\begin{bmatrix}
         -2 & -2+2i\sqrt{2} & 3-i\sqrt{2} \\
         2+2i\sqrt{2} & 5 & -4-i\sqrt{2} \\
         3+i\sqrt{2} & 4-i\sqrt{2} & -3 \\
        \end{bmatrix},\]
\[A_{11,5} = \begin{bmatrix}
         -4 & -2i\sqrt{2} & 1-i\sqrt{2}\\
         4+2i\sqrt{2} & -1+2i\sqrt{2} & -2+i\sqrt{2} \\
         3+i\sqrt{2} & i\sqrt{2} & -1+i\sqrt{2} \\
        \end{bmatrix},
   A_{11,6} =\begin{bmatrix}
         3i\sqrt{2} & 2-2i\sqrt{2} & -1-i\sqrt{2} \\
         2+2i\sqrt{2} & -1-2i\sqrt{2} & -2 \\
         3-i\sqrt{2} & -2 & i\sqrt{2} \\
        \end{bmatrix},\]
\[A_{11,7} = \begin{bmatrix}
         1+3i\sqrt{2} & -4+3i\sqrt{2} & 1-3i\sqrt{2} \\
         2+i\sqrt{2} & -1+2i\sqrt{2} & -2i\sqrt{2} \\
         3-i\sqrt{2} & 4+2i\sqrt{2} & -4 \\
        \end{bmatrix},
A_{11,8} = \begin{bmatrix}
         i\sqrt{2} & 0 & 1 \\
         2 & -1 & -i\sqrt{2} \\
         3-i\sqrt{2} & i\sqrt{2} & -1-i\sqrt{2} \\
        \end{bmatrix},\]
\[A_{11,9} =\begin{bmatrix}
         -2+i\sqrt{2} & -2i\sqrt{2} & 1-i\sqrt{2} \\
         4 & -3+2i\sqrt{2} & -2 \\
         3-i\sqrt{2} & -2+2i\sqrt{2} & -2 \\
        \end{bmatrix},
      A_{11,10} = \begin{bmatrix}
         -1+3i\sqrt{2} & -2-i\sqrt{2} & -3-i\sqrt{2} \\
         4-i\sqrt{2} & -1+2i\sqrt{2} & 2i\sqrt{2} \\
         3-i\sqrt{2} & 2i\sqrt{2} & 2i\sqrt{2} \\
        \end{bmatrix},\]
\[A_{12,1} =\begin{bmatrix}
         -1-i\sqrt{2} & i\sqrt{2} & 1 \\
         -2+2i\sqrt{2} & 3 & -i\sqrt{2} \\
         2+2i\sqrt{2} & 2-2i\sqrt{2} & -1-i\sqrt{2} \\
        \end{bmatrix},
        A_{12,2} = \begin{bmatrix}
         -3 & 2+2i\sqrt{2} & 2+i\sqrt{2} \\
         -2+2i\sqrt{2} & 3 & 2 \\
         2+2\sqrt{2} & 2-2i\sqrt{2} & 1-2i\sqrt{2} \\
        \end{bmatrix},\]
\[A_{12,3} = \begin{bmatrix}
         1+i\sqrt{2} & -i\sqrt{2} & -1 \\
         2 & -1 & i\sqrt{2} \\
         2-2i\sqrt{2} & -2 & 1+i\sqrt{2} \\
        \end{bmatrix},
   A_{12,4} =\begin{bmatrix}
         3+2i\sqrt{2} & -4+i\sqrt{2} & -3 \\
         2 & -1+2i\sqrt{2} & -2+i\sqrt{2} \\
         2-2i\sqrt{2} & 2+2i\sqrt{2} & 1+2i\sqrt{2} \\
        \end{bmatrix},\]
\[A_{16,1} = \begin{bmatrix}
         -1 & -i\sqrt{2} & 1 \\
         2i\sqrt{2} & -3 & -i\sqrt{2} \\
         4 & 2i\sqrt{2} & -1 \\
        \end{bmatrix},
   A_{16,2} =\begin{bmatrix}
         -1+i\sqrt{2} & 2-i\sqrt{2} & -1-2i\sqrt{2} \\
         2i\sqrt{2} & 1-2i\sqrt{2} & -4-i\sqrt{2} \\
         4 & -4-2i\sqrt{2} & -3+3i\sqrt{2} \\
        \end{bmatrix}, \]
\[ A_{16,3} = \begin{bmatrix}
         -1+2i\sqrt{2} & 2 & -2-i\sqrt{2} \\
         2i\sqrt{2} & 1 & -2 \\
         4 & -2i\sqrt{2} & -1+2i\sqrt{2} \\
        \end{bmatrix},
A_{16,4} = \begin{bmatrix}
         -1+3i\sqrt{2} & 4-3i\sqrt{2} & -1-3i\sqrt{2} \\
         2i\sqrt{2} & 1-2i\sqrt{2} & -2-i\sqrt{2} \\
         4 & -4-2i\sqrt{2} & -3+i\sqrt{2} \\
        \end{bmatrix},\]
\[A_{16,5} =\begin{bmatrix}
         -3 & 4+i\sqrt{2} & 3-i\sqrt{2} \\
         4+2i\sqrt{2} & -3-4i\sqrt{2} & -4-i\sqrt{2} \\
         4 & -4-2i\sqrt{2} & -3 \\
        \end{bmatrix},
      A_{16,6} = \begin{bmatrix}
         -3+i\sqrt{2} & 2+i\sqrt{2} & -1-3i\sqrt{2} \\
         4+2i\sqrt{2} & 1-2i\sqrt{2} & -4+3i\sqrt{2} \\
         4 & -2i\sqrt{2} & -1+3i\sqrt{2} \\
        \end{bmatrix},\]
\[A_{16,7} =\begin{bmatrix}
         -3+2i\sqrt{2} & 3i\sqrt{2} & -1-3i\sqrt{2} \\
         4+2i\sqrt{2} & 5 & -4+i\sqrt{2} \\
         4 & 4-2i\sqrt{2} & -3+2i\sqrt{2} \\
        \end{bmatrix},
        A_{16,8} = \begin{bmatrix}
         -3+3i\sqrt{2} & 4+i\sqrt{2} & -1-2i\sqrt{2} \\
         4+2i\sqrt{2} & 1-2i\sqrt{2} & -2+i\sqrt{2} \\
         4 & -2i\sqrt{2} & -1+i\sqrt{2} \\
        \end{bmatrix}\]
     \normalsize
     
\end{section}
\end{appendices}

\end{document}